\documentclass[12pt,reqno]{amsart}
\usepackage[foot]{amsaddr}

\usepackage{amssymb,mathabx,wasysym}
\usepackage{graphicx}
\usepackage[margin=1in]{geometry}
\usepackage{enumerate}
\usepackage[dvipsnames]{xcolor}
\usepackage[colorlinks,
    citecolor=Green!80!black,
    linkcolor=Green!80!black
]{hyperref}
\usepackage[font=footnotesize]{caption}
\usepackage{subcaption}

\newtheorem{thm}{Theorem}[section]
\newtheorem{lem}[thm]{Lemma}
\newtheorem{cor}[thm]{Corollary}
\newtheorem{ppn}[thm]{Proposition}
\theoremstyle{definition}
\newtheorem{dfn}[thm]{Definition}
\newtheorem{rmk}[thm]{Remark}

\newcommand{\E}{\mathbb{E}}
\renewcommand{\P}{\mathbb{P}}
\newcommand{\R}{\mathbb{R}}
\newcommand{\N}{\mathbb{N}}
\renewcommand{\emptyset}{\varnothing}
\newcommand{\set}[1]{\{#1\}}
\newcommand{\seq}[1]{\left(#1\right)}
\newcommand{\Ind}[1]{\mathbf{1}\{#1\}}
\renewcommand{\vec}[1]{\underline{#1}}

\newcommand{\pd}{\partial}
\newcommand{\tmix}{t_\textup{mix}}
\newcommand{\tcoup}{t_\textup{coup}}
\newcommand{\dtv}[1]{\|#1\|_\textup{TV}}

\newcommand{\bX}{\textbf{X}}
\newcommand{\bY}{\textbf{Y}}
\newcommand{\usi}{{\vec{\sigma}}}
\newcommand{\uta}{{\vec{\tau}}}

\newcommand{\tG}{\tilde{G}}
\newcommand{\tV}{\tilde{V}}
\newcommand{\tGF}{\tilde{G}_F}

\newcommand{\proj}{\textup{proj}}

\newcommand{\bN}{\mathbb{N}}
\newcommand{\tgm}{\tilde{\gamma}}
\newcommand{\tGm}{\tilde{\Gamma}}
\newcommand{\tE}{\tilde{\E}}

\newcommand{\tI}{\textup{I}}
\newcommand{\tII}{\textup{II}}
\newcommand{\tg}{\textup{g}}
\newcommand{\tm}{\textup{m}}
\newcommand{\Ig}{I_\textup{g}}
\renewcommand{\Im}{I_\textup{m}}

\newcommand{\tc}{\textup{c}}
\newcommand{\tpd}{\textup{d}}
\newcommand{\tv}{\tilde{v}}
\newcommand{\ta}{\tilde{a}}
\newcommand{\tlt}{\tilde{t}}
\newcommand{\Ic}{I_\textup{c}}

\newcommand{\A}{\mathsf{A}}
\newcommand{\BB}{\mathsf{B}}
\newcommand{\CC}{\mathsf{C}}
\newcommand{\DD}{\mathsf{D}}
\newcommand{\EE}{\mathsf{E}}
\newcommand{\FF}{\mathcal{F}}
\newcommand{\MM}{\mathsf{M}}
\newcommand{\NN}{\mathsf{N}}
\newcommand{\GG}{\textup{\textsf{G}}}

\newcommand{\cyc}{\mathsf{cyc}}

\newcommand{\bpp}{}

\newcommand{\ehl}{\color{black}}

\title{Rapid mixing of hypergraph independent sets}

\author[J.~Hermon]{Jonathan Hermon$^{\ddagger}$}
\email{jonathan.hermon@statslab.cam.ac.uk}
\author[A.~Sly]{Allan Sly$^{\dagger,*}$}
\email{sly@berkeley.edu}
\author[Y.~Zhang]{Yumeng Zhang$^*$}
\email{ymzhang@berkeley.edu}

\address{{Princeton University$^\dagger$}, {University of Cambridge, Cambridge, UK, Financial support by
the EPSRC grant EP/L018896/1$^{\ddagger} $ } and 
{University of California, Berkeley$^*$}
}

\begin{document}

\begin{abstract}
We prove that the the mixing time of the Glauber dynamics for sampling independent sets on $n$-vertex $k$-uniform hypergraphs is $O(n\log n)$ when the maximum degree $\Delta$ satisfies $\Delta \leq c 2^{k/2}$, improving on the previous bound~\cite{bordewich2006stopping} of $\Delta \leq k-2$.  This result brings the algorithmic bound to within a constant factor of the hardness bound of~\cite{Bezakova2015} which showed that it is NP-hard to approximately count independent sets on hypergraphs when $\Delta \geq 5 \cdot 2^{k/2}$.
\end{abstract}
\maketitle

\section{Introduction} 

We consider the mixing time of the Glauber dynamics for sampling from uniform independent sets on a $k$-uniform hypergraph (i.e.,~all hyperedges are of size $k$).  In doing so we extend the region where there is a \bpp fully polynomial-time randomized approximation scheme \ehl(FPRAS) for approximately counting independent sets, reducing an exponential multiplicative gap to a constant factor.

In the case of graphs the question of approximately counting and sampling independent sets is already well understood.  In a breakthrough paper, Weitz~\cite{weitz2006counting} constructed an algorithm which approximately counts independent sets on 5-regular graphs by constructing a tree of self-avoiding walks to calculate marginals of the distribution.  These can be approximated efficiently because of decay of correlation giving rise to a \bpp fully polynomial-time approximation scheme (FPTAS) \ehl for the problem.  This was shown to be  tight~\cite{sly2010computational} via a construction based on random bipartite graphs, proving that it is NP-hard to approximately count independent sets on 6-regular graphs.  The key difference between 5 and 6 is that on the infinite 5-regular tree, there is exponential decay of correlation of random independent sets while long range correlations are possible on the 6-regular tree.

In terms of statistical physics the difference is that there is an unique Gibbs measure on the $\Delta$-regular tree for $\Delta \leq 5$ but the existence of multiple Gibbs measures when $\Delta \geq 6$.  This paradigm extends more broadly to other spin systems such as the hardcore model (a model of weighted independent sets) and the anti-ferromagnetic Ising model.  In both cases a similar construction to \cite{weitz2006counting,sly2010computational} shows that it is NP-hard whenever these models have non-uniqueness \cite{sly2012computational} and Weitz's algorithm gives an FPTAS~\cite{sinclair2014approximation} in the uniqueness case except for certain critical boundary cases.  Together with work of Jerrum and Sinclair~\cite{jerrum1993polynomial} the problem of approximately counting in 2-spin systems on regular graphs is essentially complete.

For hypergraphs, however, even in two spin systems the question remains wide open.
A hypergraph $H=(V,F)$ consists of a vertex set $V$ and a collection $F$ of vertex subsets, called the hyperedges. An independent set of $H$ is a set $I\subseteq V$ such that no hyperedge $a\in F$ is a subset of $I$. The \bpp natural analogy \ehl with graphs would predict that the threshold for approximate counting should correspond to the uniqueness threshold for the $\Delta $-regular tree which corresponds to $\bpp \Delta = (1+o(1)) \frac1{2k} 2^k$.  This turns out to be false and in fact~\cite{Bezakova2015} showed that it is NP-hard to approximately count independent sets when $\bpp \Delta \geq 5 \cdot 2^{k/2}$.  What breaks down is that Weitz's argument requires not just exponential decay of correlation but also a stronger notion known as strong spatial mixing (SSM) which fails to hold when $\Delta  \geq 6$ for all $k\geq 2$~\cite{Bezakova2015}.

Despite the lack of SSM, Bez\'{a}kov\'{a}, Galanis, Goldberg, Guo, \v{S}tefankovi\v{c}~\cite{Bezakova2015} were able to give a modified analysis of the Weitz's tree of self avoiding walks algorithm and gave a deterministic FPTAS for approximating the number of independent sets when $200 \leq \Delta  \leq k$.  In this paper we study the Glauber dynamics where previously using path coupling Bordewich, Dyer and Karpinski~\cite{bordewich2008path,bordewich2006stopping} showed that the mixing time is $O(n\log n)$ when $\Delta  \leq k-2$ (where throughout $n$ is used to denote the number of vertices).  These bounds, while holding for larger $\Delta$ than the graph case, still fall far short of $5 \cdot 2^{k/2}$, the hardness bound. 
Our main result gives an improved analysis of the Glauber dynamics narrowing the computational gap to a multiplicative constant.  In the case of {\it linear} hypergraphs, those in which no hyper-edges share more than one vertex, much stronger results are possible.

\begin{thm}\label{t:mixing}
There exists an absolute constant $c>0$ such that for every $n$-vertex hypergraph $G$ with edge size $k$ and maximal degree $\Delta$, the Glauber dynamics mixes in $O(n\log n)$ time if the graph satisfies one of the following conditions:
    \begin{enumerate}[1.]
        \item $ \Delta \le c2^{k/2}$.
        \item $ \Delta \le c2^{k}/k^2$ and $G$ is linear.
    \end{enumerate}
\end{thm}

We expect our approach to hold \bpp when the sizes of hyperedges are \ehl at least $k$ but for the sake of simplicity of the proof we restrict our attention to the case of \bpp constant hyperedge size\ehl.  When the hypergraph is linear we achieve a much stronger bound of $\Delta \le c2^{k}/k^2$, close to the uniqueness threshold of $\bpp \Delta = (1+o(1)) \frac1{2k} 2^k$.  This suggests the  possibility that it may only be the presence of hyperedges with large overlaps that is responsible for the discrepancy with the tree uniqueness threshold. Indeed, in the hardness \bpp construction \ehl of~\cite{Bezakova2015} pairs of hyperedges have order $k$ vertices in common.

Our mixing time proof directly translates into an algorithm for approximately counting independent sets.

\begin{cor}\label{t:FPRAS}
    There is an FPRAS for counting the number of hypergraph independent sets  for all hypergraphs with maximal degree $\Delta $ and edge size $k$ satisfying the conditions of Theorem~\ref{t:mixing}.
\end{cor}

Closely around the time a first version of this manuscript was posted on arXiv, two related results were posted by different groups of authors: Moitra~\cite{moitra2016approximate} gave a new FPTAS up to $\Delta  \leq 2^{k/20}$; Heng, Jerrum and Liu~\cite{guo2016uniform} gave an \emph{exact} sampling algorithm that has $O(n)$-average running time when $\Delta \le \frac{1}{\sqrt{6e}k}2^{k/2}$ and the minimum intersection $s\ge O(\log \Delta+\log k)$ between any pair of hyperedges.
Both results were inspired by the recent breakthroughs on algorithmic Lov\'{a}sz Local Lemma \cite{moser2009constructive}, but took significantly different approaches beyond that.
While not giving as sharp results as ours in the case of hypergraph independent sets (i.e.,~monotone CNF formulas), the two algorithms apply to general CNF formulas. 
It also worth noticing that while our algorithm works better when neighbouring hyperedges have small intersections, the algorithm in \cite{guo2016uniform} works best when the intersections are large.
\smallskip

We also consider the case of a random regular hypergraph which is of course locally treelike. Let $\mathcal{H}(n,d,k)$ be the uniform measure over the set of hypergraphs with  $n$ vertices,  degree $d$ and edge size $k$.  In this case we are able to prove fast mixing for $d$ growing as $c 2^{k}/k$ which is the same asymptotic as the uniqueness threshold.
\begin{thm}
\label{t:mixing_regular}
    There exists an absolute constant $c>0$ such that if $H$ is a random hypergraph sampled from $\mathcal{H}(n,d,k)$ with $d\le c2^k/k$, then with high probability (over the choice of $H$), the Glauber dynamics mixes in $O(n\log n)$ time.
\end{thm}

The only property of random regular hypergraphs used in the proof of Theorem \ref{t:mixing_regular} is that  there exists some constant $N \ge 1$ such that each ball of radius $R= R(\Delta,k) \equiv \lceil 3 e \Delta k^2 \rceil $ contains  at most $N$ cycles.
Indeed, if $H \sim \mathcal{H}(n,d,k)  $ then this holds for  $N=1$ with high probability for every fixed radius (cf. \cite[Lemma 2.1]{lubetzky2010cutoff}).
\subsection{Proof Outline}

As noted above two key methods for approximate counting, tree approximations and path coupling break down far from the computational threshold.  Tree approximations rely on a strong notion of decay of correlation, strong spatial mixing, which as noted above breaks down even for constant sized $\Delta$ and the work Bez\'{a}kov\'{a} et al.~\cite{Bezakova2015} to extend to $\Delta $ growing linearly in $k$ required a very detailed analysis.  Similarly, for the Glauber dynamics, path coupling also breaks down for linear sized $\Delta$~\cite{bordewich2006stopping}.

It is useful to consider the reasons for the limitations of path coupling.  Disagreements can only be propagated when there is a hyperedge with $k-1$ ones. However, such hyperedges should be very rare. Indeed, we show that in equilibrium the probability that a certain hyperedge $a$ has $k-1$ ones is at most $(k+1) 2^{-k}$. Thus when $\Delta$ is small, most vertices will be far from all such hyperedges which morally should give a contraction in path coupling.  However, in the standard approach of path coupling we must make a worst case assumption of the neighbourhood of a disagreement.

Our approach is to consider the geometric structure of bad regions in space time $V \times \R_+$.  In Section~\ref{s:weakver} we give a simpler version of the proof which loses only a polynomial factor in the bound on $\Delta$, yet highlights the key ideas that will be used later.  We bound the bad space time regions via a percolation argument  showing that if coupling fails, then some disagreement at time 0 must propagate to the present time, which corresponds to a vertical crossing. This geometric approach is similar to the approach of Information Percolation used to prove cutoff for the Ising model~\cite{lubetzky2015information} and avoids the need to assume worst case neighbourhoods.

In Section~\ref{s:weakver} we control the propagation of disagreements by discretizing the time-line into blocks of length $k$ and considering some (fairly coarse) necessary conditions for the creation of new disagreements during the span of an entire block of time.  In particular, in order for a new disagreement to be created at $a \in F$ during a given time interval $I$, there must be some $t \in I$ at which the configuration on \bpp the vertices of \ehl $a$ has $k-1$ ones.   This allows us to exploit the independence between different time blocks.    The proof for the sharp result is given in Section~\ref{s:generalG} and Section~\ref{s:regularG} where a more refined analysis is carried out. The main additional tool is to find an efficient  scheme for controlling  the propagation of disagreements via an auxiliary continuous time process so that we again can exploit independence,  as well as certain positive correlations associated with that process.

Finally, in Section \ref{s:FPRAS} we present the proof of Corollary~\ref{t:FPRAS} via a standard reduction from sampling to counting.

\section{Preliminaries}
\subsection{Definition of model}
In what follows it will be convenient to treat vertices and hyperedges in a uniform manner, for which reason  we consider the bipartite graph representation $G=(V,F,E)$ of a hypergraph $H = (V,F)$, where $V$ is the set of vertices, $F$ is the set of hyperedges, and $E=\{ (v,a): v \in a \in F  \} $ (i.e.,~we connect vertex $v\in V$ to hyperedge $a\in F$ if and only if $v$ appears in hyperedge $a$).
\bpp Let $n\equiv |V|$ denote the number of vertices in $G$. \ehl
For each $v\in V$, we will denote by $\pd v$ the neighbours of $v$ in $G$, which is a subset of $F$ and for each $a\in F$ define $\pd a$ similarly.
Under this notation, the degree of a vertex $v$ equals $|\pd v|$ while the size of a hyperedge $a$ equals $|\pd a|$.

An \emph{independent set} of hypergraph $G$ can be encoded as a configuration $\usi \in \set{0,1}^V$ satisfying that for every $a\in F$, there exists $v\in \pd a$ such that $\sigma(v) \neq 1$. We denote by $\Omega\equiv \Omega(G)\subset 2^{V}$ the set of all such configurations and consider the uniform measure over $\Omega$ given by
\[
    \pi(\usi) = \frac{1}{Z(G)}\Ind{\textup{$\usi$ is an independent set of $G$}},
\]
where the \bpp normalizing constant $Z\equiv Z(G) \equiv |\Omega(G)|$ counts the number of hypergraph independent sets on $G$.\ehl 
\bigskip

The (discrete-time) \emph{Glauber dynamics} on the set of independent sets is the Markov chain $(W_t)_{t\ge 0}$ with state space $\Omega$ defined as following: \bpp For each configuration $\usi\in\Omega$, vertex $v\in V$ and binary variable $x\in\set{0,1}$, let $\usi^{v,x}$ be the configuration that equals $x$ at vertex $v$ and agrees with $\usi$ elsewhere. 
Suppose that the Markov chain is at state $W_t=\usi$ at time $t$. The state at time $t+1$ is then determined by uniformly selecting a vertex $v\in V$ and performing the following update procedure:\ehl
\begin{enumerate}[\quad 1.]
    \item With probability $1/2$, set $W_{t+1} = \bpp \usi^{v,0}$.
\item With \bpp the rest \ehl probability $1/2$,  set $W_{t+1} =\bpp  \usi^{v,1}$ if $\bpp \usi^{v,1}\in\Omega$ and $W_{t+1} =\usi$ otherwise.
(If the latter case happens then $W_{t+1}(v) =\sigma(v) = 0$.)
\end{enumerate}
This Markov chain is easily shown to be ergodic with stationary distribution $\pi$. Its (total variation) mixing time, denoted by $\tmix$, is defined to be
\[
    \tmix(\epsilon) \equiv
    \inf\set{t:\max_{\usi\in\Omega}
    \dtv{\P(W_t=\cdot \mid W_0=\usi) - \pi(\cdot)} < \epsilon}
    ,\quad
    \tmix \equiv \tmix(1/4),
\]
where $\|\mu-\nu \|_{\mathrm{TV}}\equiv\frac{1}{2} \sum_{\usi\in\Omega}|\mu(\usi)-\nu(\usi)| $. In what follows it is convenient to consider the continuous-time Glauber dynamics $X_t$ defined as follows. Place at each site $v\in V$ an i.i.d.\ rate-one Poisson clock; at each clock ring, we update the associated vertex in the same manner as in the discrete-time chain. The mixing time of the continuous chain $X_t$ can be defined similarly and we denote it by $\tmix^\textup{ct}$. It is well-known (cf.\ \cite[Thm.~20.3]{levin2009markov}) that the two mixing times $\tmix,\tmix^\textup{ct}$ satisfy the following relation.
\begin{ppn}\label{p:tmixtmix}
    Under the notation above, $\tmix(\epsilon) \le 4|V| \tmix^\textup{ct}(\epsilon/2)$.
\end{ppn}

\subsection{Update sequence and grand coupling}

The \emph{update sequence} along an interval ${(}t_0,t_1]$ is the set of tuples of the form $(v,t,U)$, where $v\in V$ is the vertex to be updated, $t\in {(}t_0,t_1]$ is the update time, and $U\in \set{0,1}$ is the tentative update value of $v$ (``tentative" as it might be an illegal update). An update $(v,t,U)$ is said to be \emph{blocked} (by hyperedge $a$) in configuration $\usi$, if $U = 1$ and there exists $a\in \pd v$ such that $\sigma(u) = 1$ for all $u \in \pd a\setminus\set{v}$.

Under this notation, the update rule of the Glauber dynamics can be rephrased as updating the spin at $v$ to $U$ at time $t$ unless the update is blocked in $X_t$. Therefore $X^\usi_t$, the continuous-time chain starting from initial configuration $\usi$, can be expressed as a deterministic function of $\usi$ and the update sequence $\xi$. We will denote this function by $\bX$ and write
\[
    X^{\usi}_{t} = \bX[\xi, \usi;0,t].
\]
We remark that $\bX$ depends on the underlying graph $G$ implicitly.

A related update function $\bY$ is given by setting the spin at $v$ to $U$ at each update $(v,t,U)$  \emph{regardless of} whether the update is blocked or not. We define a family of processes $Y_{s,t}$ on the state space of $2^V$ which, given the all-one initial configuration $\vec{1}$ and the update sequence $\xi_{s,t}$ along the interval ${(}s,t]$, satisfies
\[
    Y_{s,t}  \equiv \vec{1},
    \textup{ if  } t\le s,
    \quad \textup{and }
    Y_{s,t} \equiv \bY[\xi_{s,t}, \vec{1};s,t],
    \textup{ if  } t > s.
\]

In the continuous-time setting, the update sequence $\xi\equiv\seq{(v_\ell,t_\ell,U_\ell)}_{\ell\ge 1}$ follows a marked poisson process  where  $\xi^\circ \equiv \seq{(v_\ell,t_\ell)}_{\ell\ge 1}$ is a poisson point process on $V\times{[}0,\infty)$ with rate $1$ (per site) and $\seq{U_\ell}_{\ell\ge 1}$ is a sequence of i.i.d.\ Bernoulli($1/2$) random variables independent of $\seq{(v_\ell,t_\ell)}_{\ell\ge 1}$.
\bpp Using the same marked Poisson process $\xi$ for different update functions\ehl,  the discussion above provides a \emph{grand coupling} of the processes $(X^{\usi}_t)_{t \ge 0}$ and $(Y_{s,t})_{t \ge 0}$ \bpp for all possible values of $s,t, \usi$ simultaneously. \ehl 

{\bpp It is straightforward to check that, under the above setting, for any fixed $s\ge 0$ the process $(Z^{(s)}_t)_{t\ge 0}$, where $Z_t^{(s)}\equiv Y_{s,s+t} $, is a continuous-time simple random walk on the hypercube $\set{0,1}^V$ with initial state $Y_{s,s} = \vec 1$, in which each co-ordinate is updated at rate 1.}  

The purpose of introducing $Y_{s,t}$ is to utilize the monotonicity in the constraints of independent set and provide a uniform upper bound to $X^\usi_t$ for all $\usi\in \Omega$.
\bpp For simplicity of notation, we write $Y_t\equiv Y_{0,t}$.  \ehl
For any pair of vectors $X,Y$ in $\set{0,1}^V$, write $X\le Y$ if and only if $X(v)\le Y(v)$ for all $v\in V$.
\begin{ppn}\label{p:monotonicity}
    Under the notations above, for all $\usi\in\Omega(G)$ and $s,t\ge 0$, we have
\begin{equation}\label{e:monotonicity}
    X^\usi_t\le Y_{t} \le Y_{s,t}.
\end{equation}
\end{ppn}
\begin{proof}
    \bpp We proceed by showing that \eqref{e:monotonicity} holds for each configuration $\sigma\in \Omega(G)$ and update sequence $\xi=\seq{(v_\ell,t_\ell,U_\ell)}_{\ell\ge 1}$.
By the right continuity of the process, it is enough to verify \eqref{e:monotonicity} at time $0$ and the times of updates $(t_\ell)_{\ell\ge 1}$.  When referring to the second inequality, we may assume in addition that $s\le t$, as otherwise  $Y_{s,t} = \vec{1}$ and thus there is nothing to prove. 

At time $t_0 \equiv 0$, \eqref{e:monotonicity} holds  since $\usi\le \vec{1}$ for all $\usi\in\Omega(G)$.
Suppose by induction we have verified \eqref{e:monotonicity} at all update times $(t_{\ell'})_{\ell'\le \ell-1}$. Since there is no update between $t_{\ell-1}$ and $t_{\ell}$, \eqref{e:monotonicity} remains true till the moment immediately before $t_\ell$. At time $t=t_\ell$, the inequality is preserved if we successfully update $v_\ell$ to $U_\ell$ in each of the configurations ${X^\usi_{t}, Y_{t}, Y_{s,{t}}}$. If the update fails in some of the configurations, then $(v_\ell,t_\ell,U_\ell)$ must be blocked in $\lim_{\epsilon\downarrow 0} X^\usi_{t-\epsilon}$. In which case $U_\ell = 1$ and we set $X^\usi_{t}(v)$ to be $0$, while $\bpp Y_{t}(v)$ and $Y_{s,t}(v)$ to be $1$, again preserving the inequality. Combining the two cases together complete the induction hypothesis the $\ell$'th update.
\end{proof}

Let $\tcoup \equiv \tcoup(\xi)$ be the time the grand coupling succeeds under updating sequence $\xi$:
\[
    \tcoup \equiv
    \min\set{t: \forall \usi,\uta \in\Omega
        , X^{\usi}_t = X^{\uta}_t}.
\]
A standard argument (cf.~\cite[Thm.~5.2]{levin2009markov}) implies that for all $t>0$
\begin{equation}\label{e:dtv_coupling}
    \max_{\usi\in\Omega} \dtv{P^t(\usi,\cdot) - \pi}
    \le
    \P(\tcoup>t),
\end{equation}
where $P^t(\usi , \cdot)$ is the distribution at time $t$ of the continuous-time chain, started from  $\usi$.
\subsection{Discrepancy sequence and activation time}\label{s:discrep}

In this section we take a closer look at the update process backward in time and in particular show how a discrepancy at time $t$ can be traced back to discrepancies at earlier times. This will provide a \bpp necessary \ehl condition for $\set{\tcoup>t}$.

Given an update sequence $\xi$, a vertex $v\in V$ and $t>0$, let $t_{-}(v,t) \equiv t_{-}(v,t;\xi)$ be the time of the last update in $\xi$ at $v$ before time $t$. More explicitly, we define
\[
    t_{-}(v,t) \equiv 0\vee
    \sup \set{t'<t: (v,t',0) \in \xi
    \textup{ or }(v,t',1) \in \xi}.
\]
Fix an update sequence $\xi$ and time $t>0$. If $\tcoup = \tcoup(\xi)>t$, then there must exist two initial configurations $\usi,\uta\in\Omega$ and a ``discrepancy" $v_0$ at time $t$, i.e.,\ a vertex $v_0\in V$ such that $X^\usi_t(v_0) \neq X^\uta_t(v_0)$. Now we choose an \bpp arbitrary \ehl such discrepancy $v_0$, look at the last update of $v_0$ before time $t_1 \equiv t$ and denote its time by $t_{0}\equiv t_{-}(v_0,t_1)$.

Assume without loss of generality that  $X^\usi_{t_{1}}(v_0)=0 \neq 1= X^\uta_{t_{1}}(v_0)$. To end up with a discrepancy at $v_0$ after the update at $t_0$, \bpp its tentative update value must be $1$ \ehl and it must be blocked in $X^\usi_{t_0} $ but not in $X^\uta_{t_0}$. Hence there must exists a hyperedge $a_{0} \in \pd v_0$  such that
\[
    X^\usi_{t_0}(\pd a_0\setminus\set{v_0}) = \vec{1} \neq X^\uta_{t_0}(\pd a_0\setminus\set{v_0}).
\]
Consequently, there must exist at least one vertex $u\in\pd a_{0}\setminus \set{v_0}$  at which the two configurations disagree at time $t_0$, namely,
\[X^\usi_{t_0}(u) = 1 \neq 0 = X^\uta_{t_0}(u).\]
We arbitrarily choose one such vertex $u \in\pd a_{0}\setminus \set{v_0}$ and denote it by $v_{-1}$.

Now apply the same reasoning for the update at $v_{-1}$ at time $t_{-1} \equiv t_{-} (v_{-1},t_{0})$. We can find a hyperedge $a_{-1}\in\pd v_{-1}$ blocking the update $(v_{-1},t_{-1},1) \in \xi $ in exactly one of the two configurations $X^\usi_{t_{-1}}$ and $X^\uta_{t_{-1}}$ (namely, in the latter). Moreover, there must exists a discrepancy at a certain vertex $v_{-2}\in\pd a_{-1}$ at time $t_{-1}$.
Repeating the process until time $0$ produces a sequence of tuples $\zeta^\circ \equiv \seq{(v_{-\ell},t_{-\ell},a_{-\ell})}_{0 \le \ell \le L}$, where $(t_\ell)_{-L \le \ell \le 0}$ satisfies that
\[
    \bpp
    t_1 = t,\quad t_{-L} = 0\quad
    \textup{ and for all }  0 \le \ell < L, \quad t_{-\ell} \equiv t_{-}(v_{-\ell}, t_{-(\ell-1)})\le t_{-(\ell-1)};
\]
$(v_{-L},a_{-L}) $ satisfies that $a_{-L}=a_{-L+1}$ and 
\begin{equation}
\label{e: minmax}\bpp
X^\usi_{t_{-L} }(v_L) = X^\usi_{0}(v_L)\neq  X^\uta_{0}(v_L) =X^\uta_{t_{-L}}(v_L);
\end{equation}
and for each $0 \le \ell < L$:
\begin{enumerate} [\quad 1.]
    \item The update $(v_{-\ell},t_{-\ell},1)$ exists in $\xi$.
    \item The hyperedge $a_{-\ell}$ contains $v_{-\ell}$ and $v_{-(\ell+1)}$  and the update $(v_{-\ell},t_{-\ell},1)$ is blocked by $a_{-\ell}$ in exactly one of the two configurations $X^\usi_{t_{-\ell}}$ and $X^\uta_{t_{-\ell}}$.     \item The vertex $v_{-\ell}$ is not updated in the time interval $(t_{-\ell}, t_{-(\ell-1)}]$.
\end{enumerate}

\bpp
Condition 2 in the above description is hard to analyse directly, because  in general it is hard to control the probability that an update is blocked for the process $X^{\usi}_{t}$ at time $t$. This is where we use monotonicity (\eqref{e:monotonicity}). Observe that whenever an update $(v,t,1)$ is blocked by a hyper-edge $a$ in one of the two processes $X^\usi_t,X^\uta_t$ at time $t$, one of them must be all $1$ on  $\pd a\setminus v$, and by monotonicity so is $Y_t(\pd a\setminus v)$. Namely
\[
    \vec 1\ge Y_{t}(\pd a\setminus v) \ge \max
    \set{
        X^\usi_t (\pd a \setminus v),
        X^\uta_t (\pd a \setminus v)
    }
    =\vec 1
    .
\]
Meanwhile, since we are trying to update the value at $v$ to $1$ at time $t$, after this update (which is always successful in $Y_t$), $Y_t(v)$ equals to 1 as well. Therefore, the sequence $\xi^\circ$ satisfies that 
\begin{equation}\label{e:bad_al} \bpp
    Y_{t_{-\ell}}(\pd a_{-\ell}) = \vec 1
    , \textup{ for all }
    0\le \ell<L
    .
\end{equation}
Equation \eqref{e:bad_al} is the key property of our proof and will be used repeatedly in what comes.\ehl

 For convenience of \bpp later \ehl application, it is useful to consider also the following representation of $\zeta^\circ$ with non-negative indices, which moves forward in time rather than backwards as in the original construction of $\zeta^\circ$: \bpp Let $\zeta \equiv ((v'_{\ell},t'_{\ell},a'_{\ell}))_{0\le \ell\le L}, $ be defined as 
\[
    (v'_{\ell},t'_{\ell},a'_{\ell})\equiv (v_{\ell-L},t_{\ell-L},a_{\ell-L})
    \quad\textup{ for each } 0 \le \ell \le L,
\]
and write $t'_{L+1} \equiv t_1$ for the endpoint of the time interval. \ehl 
We will refer to such a sequence $\zeta$ as a \emph{discrepancy sequence up to time} $t_1=t$ (with respect to $\usi,\uta$ and $\xi$). It is straightforward to check that
\begin{lem}\label{l:discrep}
    Given an update sequence $\xi$ and a time $t\ge 0$, if $\set{\tcoup>t}$, then there exists a discrepancy sequence $\zeta$ up to time $t$ as defined above.
\end{lem}

We end the section with one more definition.
 \begin{dfn}\label{d:active}
     Let $t\ge 0$, $a\in F$ and  $v \in \pd a$. We say that $(v,a)$ is \emph{activated} (resp.~$s$-\emph{activated}) at time $t$ if the update sequence $\xi$ contains an update $(v,t,1)$ and $Y_{t}(\pd a)= \vec{1}$ (resp.~$Y_{s,t}(\pd a)= \vec{1}$). We say that $a$ got activated (resp.~$s$-\emph{activated}) at time $t$ if $(u,a)$ got activated (resp.~$s$-\emph{activated}) at time $t$ for some $u \in \pd a$.  We further define \bpp $(v,a)$ to be \emph{active} \ehl at time 0 for all $a \in F$, $v \in \pd a$. 
\end{dfn}

\section{A simplified proof of a weaker version}\label{s:weakver}
To illustrate the key ideas of our proof technique, we first define an auxiliary site percolation on the space-time slab of the update history, and use it to prove the following weaker version of Theorem \ref{t:mixing}.

\begin{thm}\label{t:mixing_weak}
    For every $n$-vertex hypergraph $G$ of edge size $k$ and maximal degree $\Delta $, the Glauber dynamics mixes in $O(n\log n)$ time if the graph satisfies one of the following conditions:
    \begin{enumerate}[1.]
        \item $\Delta \le 2^{k/2}/(\sqrt{8} k^{2})$.
        \item $\Delta \le 2^{k}/(9k^3)$ and $G$ is linear.
    \end{enumerate}
\end{thm}

\subsection{The auxiliary percolation process}
We break up the space-time slab into time intervals $\bpp\seq{[T_i,T_{i+1})}_{i\ge 0}$ of length $k$, where $T_i \equiv ik$, $i=0,1,\dots$
We shall neglect the possibility that a certain edge got activated precisely at some time $T_i$, as this has probability 0.
For $t>0$, define $ i(t) \equiv \lfloor t/k \rfloor$.
Let $\tGF$ be an oriented graph with vertex-set $F\times \N$ and edge-set $\tilde{E}_F$ satisfying that for any pairs of hyperedges $a,b\in F$ and integers $i,j \in \N$, there is an (oriented) edge from site $(a,i)$ to site $(b,j)$ in $\tilde{E}_F$ if and only if
\begin{equation}\label{e:tG}
    \pd a \cap \pd b \neq \varnothing
    \textup{ and } j-i \in \{0, 1\}.
\end{equation}
\begin{dfn}
\label{d: actvsuceptible}
Fix an update sequence $\xi$.  We say that a site $(a,i)$ is \emph{active} if $i = 0$ or  $a$ is $T_{i-1}$-activated at some time $t\in{[}T_{i},T_{i+1})$. We say that a site $(a,i)$ is \emph{susceptible} if there exists $v \in \pd a$ such that $v$ is not updated during the time interval ${[}T_i, T_{i+1})$.  A site $(a,i)$ is then called \emph{bad} if either it is active or  there exists $0\le j<i$ such that $(a,j)$ is active and $(a,\ell)$ is susceptible  for all $j+1\le \ell \le i $.
\end{dfn}\bpp
An example of the site percolation is given in Figure \ref{fig:siteperco1} where the underlying graph $\tGF$ is given in Figure \ref{fig:siteperco2}. \ehl 
The set of bad sites can be viewed as a site percolation on the graph $\tGF$ in which each site $(a,i)$ is open if it is bad with respect to the update sequence $\xi$.
The next lemma relates the success of the grand coupling to the existence of open paths in the aforementioned site percolation.
\begin{lem}\label{l:reduction_weak}
    For every update sequence $\xi$ satisfying $\tcoup>T_{M+1} = (M+1) k$, there exists an oriented path of sites in $\tGF$ that starts from $F\times \set{0}$, ends at $F\times \set{M}$ and satisfies that every site along the path is bad with respect to $\xi$.
\end{lem}

\begin{proof}
    Fix an update sequence $\xi$ and an integer $M\in \N$ such that $\tcoup> T_{M+1}$. By the discussion in Section \ref{s:discrep}, we can find two initial configurations $\usi,\uta\in\Omega(G)$ and a discrepancy sequence ${\zeta^\circ \equiv \seq{(v_{-\ell},t_{-\ell}, a_{-\ell})}_{0 \le \ell \le L}}$ from time $t_{-L}=0$ up to time $t_1 = T_{M+1}$ with respect to $\usi,\uta$ and $\xi$.  We proceed to construct a path $\gamma\equiv \gamma(\zeta)$ in $\tGF$ based on $\zeta^\circ$. (See Figure~\ref{fig:siteperco} for an illustration.)
\begin{figure}[t]
\begin{center}
\begin{minipage}{.8\linewidth}\centering
    \includegraphics{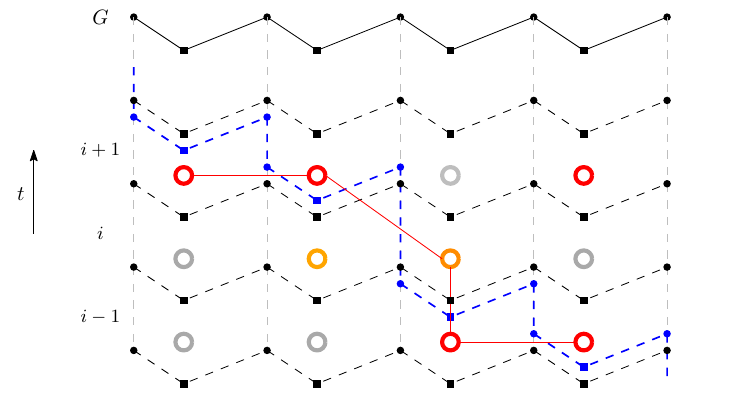}
    \subcaption{\footnotesize The discrepancy sequence projected onto $\tGF$.  }\label{fig:siteperco1}
    \smallskip

    \begin{minipage}{\linewidth}
\flushleft\footnotesize The dashed blue line represents the discrepancy sequence $\zeta$, the red (resp.\ orange, gray) circles represent active (resp.\ susceptible, other) sites in $\tGF$ and the red path represents $\gamma$. Note that the second site on the second row is susceptible but not bad since it is not immediately above any bad sites.
    \end{minipage}
    \bigskip

    \ \, \includegraphics{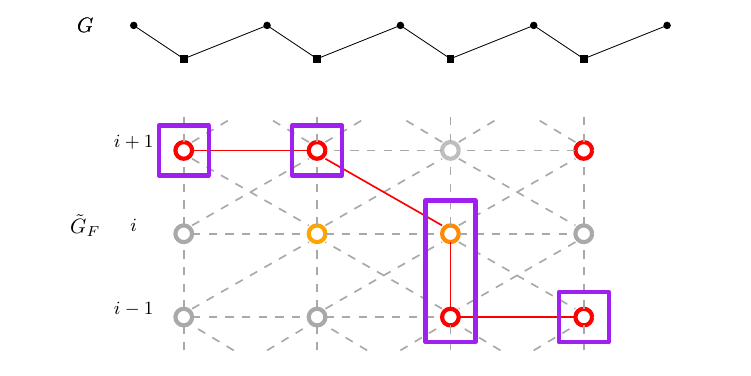}
    \subcaption{\footnotesize An open path in the site percolation on $\tGF$}\label{fig:siteperco2}
    \begin{minipage}{\linewidth}
        \flushleft\footnotesize The dashed gray lines are the underlining graph $\tGF$ where we ignore the direction of the edges. The red line is the openpath $\gamma$ constructed based on the discrepancy path $\zeta$ in panel \textsc{a}. And the purple rectangles marks the segments $\gamma_\ell$.
    \end{minipage}
\end{minipage}
\caption{Discrepancy sequence and site percolation}\label{fig:siteperco}
\end{center}
\end{figure}

    Recall that $i(t) = \lfloor t/k \rfloor $ is the \bpp time interval \ehl $t$ belongs to.
    \bpp Naturally, we would like our path $\gamma$ to pass through sites $\set{(a_{-\ell}, i(t_{-\ell}))}_{0\le \ell \le L}$, where the updates in $\zeta^\circ$ take place, and stays at each hyperedge until the next update happens at a nearby hyperedge. Let $\gamma_\ell$ denote the segment of $\gamma$ corresponding to the $\ell$'th update $(\nu_{{-\ell}},a_{-\ell},t_{-\ell})$ of $\zeta^\circ$. If the next (i.e., $(\ell-1)$'th) update happens at the same time interval as the $\ell$'th update, i.e., $i(t_{-\ell})=i(t_{-\ell+1})$, then we define $\gamma_\ell$ to be a singleton pair  $((a_{-\ell},i(t_{-\ell})))$. Otherwise if $i(t_{-\ell})<i(t_{-\ell+1})$, we define $\gamma_\ell$ as the vertical line
\[
    \gamma_\ell \equiv \gamma_\ell(\zeta^\circ) \equiv
    \seq{(a_{-\ell},j):
    i(t_{-\ell})\le j <
    i(t_{-\ell+1})}
    .
\]
\bpp Let $\gamma \equiv \cup_{0\le \ell \le L}\gamma_\ell$ be the sequential concatenation of $(\gamma_\ell)_{0\le \ell\le L}$. \ehl It is easy to observe that each $\gamma_\ell$ is a connected path in $\tGF$, $\gamma_0$ intersects $F\times \set{M}$ at $(a_0, i(t_1) - 1) = (a_0,M)$, and $\gamma_{L}$ intersects $F\times \set{0}$ at $(a_{-L},0)$. To verify the rest of the requirements of Lemma~\ref{l:reduction_weak}, we first check that every site $(a,i)\in \gamma$ is bad, distinguishing three cases:
\begin{enumerate}[\quad 1.]
    \item For $0\le \ell \le L$, if $\gamma_\ell$ is a singleton, then it must have the form $\gamma_\ell = \set{(a,i) = (a_{-\ell},i(t_{-\ell}))}$.  By condition \eqref{e:bad_al} $Y_{t_{-\ell}}(\pd a_{-\ell}) = \vec 1 $ and hence also $Y_{T_{i-1},t_{-\ell}}(\pd a_{-\ell}) = \vec 1  $, implying that  $(a,i)$ is indeed active.

 \medskip

\item  For $0\le \ell < L$, if $\gamma_\ell$ consists of more than one site, then arguing as above we know that the first site $(a_{-\ell},i(t_{-\ell}))$ is active. For each of the remaining sites $(a,i)$ with  $i(t_{-\ell})<i<i(t_{-(\ell-1)})$, rewriting the assumption gives
    \[t_{-\ell}  <   T_i < T_{i+1}  \le  t_{-(\ell-1)},\]
    i.e., $v_{-\ell}\in\pd a_{-\ell}$ is not updated during the time interval ${[}T_i,T_{i+1})$ and so $(a_{-\ell}, i)\in\gamma_\ell$ is susceptible. Following the second case of Definition~\ref{d: actvsuceptible}, $(a_{-\ell}, i)$ is bad for all   $i(t_{-\ell})<i<i(t_{-(\ell-1)})$.

\medskip

\item For each $(a_{-L},i) \in \gamma_{L}$, recall from the definition of $\zeta^{\circ}$ that $a_{-L}=a_{-(L-1)}$ and that $t_{-}(v_{-L},t_{-(L-1)}) = 0$, i.e.,\ $v_{-L}$ is never updated before $T_{i(t_{-(L-1)})}$. Hence all sites $(a_{-L},i)$ with $i<i(t_{-(L-1)})$ are susceptible. Their badness then follows from the fact that $(a_{-L},0)$ is bad \bpp (recall that $(a,0)$ is defined to be active for all $a \in F$).\ehl
\end{enumerate}

All that is left is to check that $\gamma_\ell$ is connected to $\gamma_{\ell -1}$ for each $1\le \ell \le L$. The connectivity of $\gamma_{L}$ to $\gamma_{L-1}$ is trivial. For $1\le \ell < L$, the first site in $\gamma_{\ell-1}$ is $(a_{-(\ell-1)}, i^\textup{start}_{\ell-1} \equiv i(t_{-(\ell-1)}))$ and the last site in $\gamma_\ell$ is $(a_{-\ell}, i^\textup{end}_{\ell})$ with
\[
    i^\textup{end}_{\ell} \equiv
    \begin{cases}
        i(t_{-(\ell-1)})-1
        & \textup{if } i(t_{-(\ell-1)}) > i(t_{-\ell}),\\
        i(t_{-\ell})
        & \textup{if } i(t_{-(\ell-1)}) = i(t_{-\ell})
    \end{cases}
    .
\]
In either case, site $(a_{-\ell}, i^\textup{end}_\ell)$ is connected to $(a_{-(\ell-1)},i^\textup{start}_{\ell-1} )$ in $\tGF$ since $\pd a_{-\ell} \cap \pd a_{-(\ell-1)} \supseteq \set{v_{-\ell}}\neq \varnothing$ and
\[
i^\textup{start}_{\ell-1} - i^\textup{end}_\ell = 1-\Ind{i(t_{-(\ell-1)}) = i(t_{-\ell})} \in \{0, 1\}.\]
\end{proof}

\subsection{Proof of Theorem~\ref{t:mixing_weak}}
\bpp By Proposition~\ref{p:tmixtmix} and \eqref{e:dtv_coupling}, we aim to show that under the assumptions of Theorems~\ref{t:mixing_weak}, there exists some constant $C$ such that for $M\equiv \lceil C \log n  \rceil $,  $\P[t_{\mathrm{coup}}>T_{M+1}] \le 1/8$.
In fact, as will be clear from the proof, once we find an $M$ for which our analysis yields that $\P[t_{\mathrm{coup}}>T_{M+1}] \le 1/8$, doubling it gives $\P[t_{\mathrm{coup}}>T_{2(M+1)}] \le O(n^{-1})$.  \ehl

Using Lemma \ref{l:reduction_weak}, it suffices to bound the probability that there exists an oriented path from $F \times \set{0} $ consisting of bad sites. We begin with some basic properties of the auxiliary percolation process.

\begin{ppn}\label{p:percprop}
    Let $\A(a,i)$ (resp.\ $\mathsf{S}(a,i)$) be the event that site $(a,i)$ is active (resp.\ susceptible). Then for every distinct $(a,i), (b,j)\in\tGF$ with $i\ge j\ge 1$,
\begin{align}
    &\P(\A(a,i)) \le (k^2+1) 2^{-k} + ke^{-k}
,\quad
\P(\mathsf{S}(a,i+1)) \le ke^{-k}\label{e:percprop1}
\\
&\P(\A(a,i)\mid \A(b,j))
\le (k^2+1) 2^{-k+|\pd a \cap \pd b|}
+
ke^{-k}
,\label{e:percprop2}
\end{align}
Moreover for every set of sites $S$ in $\tGF$, if for all $(b,j)\in S$, $(a,i)$ and $(b,j)$ are not connected in  $\tilde{E}_F$, then $\A(a,i), \mathsf{S}(a,i)$ are independent of the events $\set{\A(b,j),\mathsf{S}(b,j): (b,j)\in S}$.
\end{ppn}

\begin{proof}
    The second part of \eqref{e:percprop1} is simply a union bound. To show the first part of \eqref{e:percprop1}, we notice that if $\A(a,i)$ happens then \bpp so does \ehl one of the following scenarios
\begin{itemize}
\item[(a)]
    Some  $v \in \pd a$ is not updated in $ {[}T_{i-1},T_{i}) $,  namely $\mathsf{S}(a,i-1)$ happens.
\item[(b)] \bpp
    Case  (a)  fails but $a$ gets $T_{i-1}$-activated at some time $s \in {[}T_{i},T_{i+1})$.
\end{itemize}
Case (a) happens with \bpp probability at most $ke^{-k}$. \ehl
Conditioned on the failure of case (a), namely every $v \in \pd a$ being updated in $[T_{i-1},T_{i})$, $(Y_{T_{i-1},s}(u))_{u \in \pd a } $ become $\mathrm{i.i.d.} $ Bernoulli($1/2$) r.v.'s for all $T_{i} \le s < T_{i+1}$. \bpp For case (b) to happen, there must exist $u\in\pd a$ and $s \in {[}T_{i},T_{i+1})$ such that $(u,s,1) \in \xi$ and $Y_{T_{i-1},s}(\pd a)=\vec{1}$. Hence conditioned on the failure of case $a$, by Markov's inequality the probability of case (b) is at most $k^2 2^{-k}$, where the term $k^2$ represents the expected number of updates of the vertices in $\pd a$ during ${[}T_{i},T_{i+1})$ and $2^{-k}$ is the probability that $Y_{T_{i-1},s}(\pd a)=\vec{1}$ at the time of update. Combining  the two cases  gives the first half of \eqref{e:percprop1}. \ehl
    The proof of \eqref{e:percprop2} is completely analogous, where the only difference is that we first argue that
\[
    \P(\A(a,i)\mid \A(b,j)) \le \P(\A(a,i)\mid Y_{T_{i-1},s}(\pd a \cap \pd b)=\vec{1} \text{ for all }s \in {[}T_i,T_{i+1}) ).
\]
and then apply the same reasoning as before to $\pd a \setminus \pd b$.

Finally, note that the events $\A(a,i)$ and $\mathsf{S}(a,i)$ depend on $\xi$ only through $\pd a \times (T_{i-1},T_{i+1})$. Hence the independency result follows from the independency of Poisson point process.
\end{proof}

\medskip

In order to perform a first moment calculation in an efficient manner we restrict our attention to a special type of path.

\begin{dfn}
We say that an oriented path $((a_0,i_0),(a_1,i_1),\ldots , (a_r,i_r)) $ in $\tGF$ is a \emph{minimal path} if $i_0 = 0$, $i_1=1$ and for all $j_1 \le r-2$ and $j_2 \in [j_1+2,r]$ we have that $((a_{j_{1}},i_{j_{1}}),(a_{j_{2}},i_{j_{2}})) \notin \tilde E_F $. Let $\Gamma_{\mathrm{min},r}$ be the collection of all minimal paths of length $r$.
\end{dfn}
\bpp Here we require $i_1=1$ because we are primarily interested in upperbounding the number of open paths connecting $F\times \set{0}$ to $F\times{M}$ for some $M\ge 1$ and additonal steps within $F\times \set{0}$ will not help. \ehl Observe that every oriented path in $\tGF$ can be transformed into a minimal path by \bpp deleting some vertices from it. \ehl

\medskip

\begin{proof}[Proof of Theorem \ref{t:mixing_weak}, Part 1]
Fix $M\equiv \lceil C \log n  \rceil $ where the constant $C$ shall be determined later. Note that by Lemma \ref{l:reduction_weak}, if $t_{\mathrm{coup}}>T_{M+1} $, then there must be some minimal path
    $$\gamma_M = ((a_0,i_0),(a_1,i_1),\ldots , (a_M,i_M)) \in \Gamma_{\min,M}$$
consisting of bad sites. We now estimate the expected number of such paths. For brevity, we call a minimal path $\gamma$ bad if every site of $\gamma$ is bad.

Using the notations of Proposition~\ref{p:percprop}, for every $\gamma\in \Gamma_{\min,2r}$, we can write
\[
    \E [\Ind{\gamma \textup{ is bad}}]
    \le \P\Big(\bigcap_{\ell = 0}^{2r}
    (\A(a_\ell,i_\ell)
    \cup
    \mathsf{S}(a_\ell,i_\ell))
    \Big)
    \le \P\Big(\bigcap_{\ell = 1}^{r}
    (\A(a_{2\ell},i_{2\ell})
    \cup
\mathsf{S}(a_{2\ell},i_{2\ell}))
    \Big),
\]
where in the last step we discard all odd events in order to obtain the desirable independence.
Indeed, by the definition of a minimal path, for all $j< \ell$, $(a_{2\ell},i_{2\ell})$ is not connected to $(a_{2j},i_{2j})$. Hence by Proposition~\ref{p:percprop}, the events
$\set{\A(a_{2\ell},i_{2\ell})
    \cup \mathsf{S}(a_{2\ell},i_{2\ell})}_{1\le \ell\le r}$
are mutually independent and
\begin{equation}\label{e:Pbad}
    \E [\Ind{\gamma \textup{ is bad}}
    ]\le \prod_{\ell = 1}^{r}
    \big[(k^2 + 1)2^{-k} + 2ke^{-k}\big]
    \le \big(2k^22^{-k}\big)^r
\end{equation}
To conclude the proof we note that
\begin{equation}\label{e:branching}
|\Gamma_{\mathrm{min},2r}| \le n(2k(\Delta -1)+1)^{2r} \le (4k^2 \Delta^2)^r n,
\end{equation}
where the term $n$ accounts for the choice of the initial site of the path. By the above analysis the expected number of paths in $\Gamma_{\mathrm{min},2r} $ consisting of bad sites is at most $(4k^2 \Delta^2 \times 2k^22^{-k})^r n$. By our assumption that $\Delta < 2^{k/2} /(\sqrt{8} k^2 )$, we get that there exists some constant $C $ such that for $r= \lceil C \log n \rceil$ the last expectation is at most $\bpp 1/8$. This concludes the proof of part 1 of Theorem \ref{t:mixing_weak}.
\end{proof}

We now explain the necessary adaptations for the proof of part 2. In the new setup if $a \neq b$ and $\pd a \cap \pd b \neq \emptyset $, then $|\pd a \cap \pd b|=1$. Thus the event that $(a,i)$ is bad barely affects the probability that  $(b,j)$ is bad (for $a \neq b$). However, it is more challenging to control the conditional probability that $(a,i+1)$ is bad, given that $(a,i) $ is bad. To overcome this difficulty we modify the underlying graph $\tGF$ slightly.

\begin{dfn}
Let $\bar G_F $ be an oriented graph with vertex-set $F\times \N$ and edge-set $\bar{E}_F$ satisfying that for every pairs of hyperedges $a,b\in F$ and integers $i,j \in \N$, there is an (oriented) edge from site $(a,i)$ to site $(b,j)$ in $\bar{E}_F$ if and only if $a=b$ and $j=i+2$ or $a \neq b$ and
\begin{equation}\label{e:tG2}
    \pd a \cap \pd b \neq \varnothing
    \textup{ and } j-i \in \{0, 1,2\}.
\end{equation}
We say that an oriented path $((a_0,i_0),(a_1,i_1),\ldots , (a_r,i_r)) $ in $\bar G_F$ is a \emph{minimal path} if
$i_0=1$, $i_1>1$ and for all $j_1 \le r-2$ and $j_2 \in [j_1+2,r]$ we have that $((a_{j_{1}},i_{j_{1}}),(a_{j_{2}},i_{j_{2}})) \notin \bar E_F $.
Let $\bar \Gamma_{\mathrm{min},r}$ be the collection of all minimal paths of length $r$ in $\bar G_F$.
\end{dfn}

Observe that every path $\gamma$ in $\tGF$ can be transformed into a path in $\bar{G}_F$ by deleting some of its vertices. Namely, whenever we have two consecutive steps in $\gamma$ such that $((a_{\ell},i_{\ell}), (a_{\ell+1},i_{\ell+1})) \in \tilde{E}_F\setminus \bar{E}_F$, it must satisfy that $(a_{\ell+1},i_{\ell+1}) = (a_\ell,i_\ell +1)$ and one can check that in this case $((a_{\ell},i_{\ell}),(a_{\ell+2},i_{\ell+2}))\in \bar{G}_F$. By repeatedly deleting $(a_{\ell+1},i_{\ell+1})$ from $\gamma$, where $\ell$ is the minimal index such that  $(a_{\ell+1},i_{\ell+1}) = (a_\ell,i_\ell +1)$ and $(a_\ell,i_\ell ) $ has not been deleted already, we obtain a path in $\bar G_F $. \bpp From there, one can further take a subpath such that it is a minimal path in $\bar G_F$.\ehl

\begin{proof}[Proof of Theorem~\ref{t:mixing_weak}, Part 2]
    As before, if $t_{\mathrm{coup}}>T_{2(M+1)}$, then there must be some minimal path $((a_0,i_0),(a_1,i_1),\ldots , (a_M,i_M))$ in $\bar{G}_F $  consisting of bad sites. We argue that the conditional probability that $(a_{\ell},i_{\ell})$ is bad, given that $(a_0,i_0),(a_1,i_1),\ldots (a_{ \ell-1},i_{\ell-1}) $ are all bad, is at most
\[
        (1+k^2)2^{-(k-1)}+2ke^{-k} \le 3 k^2 2^{-k}
    \quad \textup{(for $k>2$)}.
\]
Indeed, by the linearity assumption we have that either $ (a_{\ell -1},i_{\ell -1} )=( a_{\ell },i_{\ell }-2)$ or  $|\pd a_{\ell -1} \cap \pd a_{\ell }|=1 $. In the first case, the same independency argument as before shows that the conditional probability that $(a_{\ell},i_{\ell})$ is also bad is at most $(1+k^2)2^{-k}+2e^{-k}$. For the second case, $(a_\ell,i_\ell)$ must be active and the desired bound follows from \eqref{e:percprop2}.

As before, we conclude by noting that  $|\bar \Gamma_{\mathrm{min},r}| \le n(3k(\Delta -1)+1)^{r} \le (3k \Delta)^r n$ and so the expected number of paths in $\bar \Gamma_{\mathrm{min},r} $ consisting of bad sites is at most $$(3k \Delta \times 3 k^2 2^{-k} )^r n=(9k^3 \Delta 2^{-k})^r n \le \bpp 1/8 \ehl, $$
provided that $r \ge C \log n  $.
\end{proof}

\section{General Hypergraphs}\label{s:generalG}
\subsection{Minimal path}
In this subsection we give the general setup for Theorem~\ref{t:mixing} and \ref{t:mixing_regular}. The first improvement from Section~\ref{s:weakver} is based on the observation that although it takes time $k$ to update all vertices of a hyperedge at least once with high probability, on average it only takes $O(1)$ time to update each vertex. This motivates us to study the propagation of discrepancies on both vertices and hyperedges in continuous time.

Recall the definition of activation from Definition~\ref{d:active}.
\bpp Observe that, in order for a discrepancy to propagate from an active hyperedge to a nearby hyperedge, the latter hyperedge must be activated before every vertex in their intersection is updated at least once, erasing all possible {dependence}. Thus we define the following continuous time analog of the time block from the previous section.
\ehl
\begin{dfn}\label{d:deactive}
For each $v \in V$, $t\ge 0$, we define the \emph{deactivation time} of $v$ (after time $t$) \bpp as the first time $v$ is updated after time $t$, namely,\ehl
\[
    T_+(v;t) \equiv
    \inf \left\{ s>t:
        (v,s,1)\in \xi \text{ or }(v,s,0)\in \xi
    \right\},
\]
For each $a,b \in F$ and  $t\ge 0$, we define the \emph{relative deactivation time} of $a$ w.r.t.\ $b$ (after time $t$) as \bpp the first time $s>t$ such that every vertex in the intersection $\pd a\cap \pd b$ is updated at least once by time $s$, namely,\ehl
\[
    T_+(a;b,t) \equiv
    \max_{v\in \pd a} T_+(v;b,t)
    ,\quad
    \textup{ where }
    T_+(v;b,t) \equiv
    \begin{cases}
        T_+(v;t) & v\in \pd b,\\
        t & v\notin \pd b
    \end{cases}
    .
\]
In particular, the deactivation time of $a$ is
$T_+(a;t) \equiv T_+(a;a,t) = \max_{v\in \pd a} T_+(v;t)$.
\end{dfn}
\bpp Under the definition above, discrepancies can only pass from one hyperedge to another before the first hyperedge is relatively deactivated w.r.t.~the latter hyperedge. In other word, (relative) deactivation time gives the time window of discrepancy propagation. \ehl

Let $\tG$ be the Cartesian product of $G$ and the time interval ${[}0,\infty)$. Namely for two sites $(w_1,t_1),(w_2,t_2)\in \tV \equiv (V\cup F) \times {[}0,\infty)$, there exists an (oriented) edge connecting $(w_1,t_1)$ to $(w_2,t_2)$ in $\tG$ if and only if
    \[
        w_1=w_2,\ \bpp t_1 \le t_2\ehl
        \quad \textup{or} \quad
        w_1 \in \pd w_2,\ t_1 = t_2.
    \]
   The set $\tG$ can be viewed as the continuous version of the space-time slab.
\begin{dfn}
\label{def:path}
Given a sequence $\gamma_L\equiv \set{(v_\ell, a_\ell,t_\ell)}_{0\le \ell \le L}$, we say that $\gamma_L$ is a path of length $L$ in $\tG$ if $t_0 =0$, $v_0\in \pd a_0, a_0  = a_1$  and for each $1\le \ell\le L$ we have that $t_{\ell-1}<t_{\ell}$, $v_\ell\in \pd a_\ell$ and $\pd a_{\ell-1} \cap \pd a_{\ell} \neq \varnothing$.   We further say that $\gamma_L$ is a path up to time $t$ to indicate that $t_L<t$.  Let $\Gamma_L$ denote the set of paths of length $L$ and $\Gamma_L(t)$ denote the subset of paths up to time~$t$.
\end{dfn}
\bpp In the previous section, we defined an auxiliary percolation on the discrete space-time slab such that every discrepancy sequence can be projected onto an open path in the percolation. Without a good analog of the percolation in continuous time, we look at the analog of ``open paths'' directly, i.e.,~paths of $\Gamma_L(t)$ that can be interpreted as the projections of discrepancy sequences.
\ehl
\begin{dfn}\label{d:minpath}
    Fix an update sequence $\xi$ and a path $\gamma_L= \set{(v_\ell, a_\ell,t_\ell)_{0\le \ell \le L}}\in \Gamma_L$. For each $1\le \ell \le L$, we say  that $(v_\ell,a_\ell, t_\ell)$,  the $\ell$th step of $\gamma_L$, is a \emph{minimal step} of $\gamma_L$ if all of the following six events hold:
\begin{align}
    &\nonumber
    \A_{\ell}
    \equiv
    \set{(v_{\ell},t_{\ell},1)\in \xi}
    ,\qquad\qquad\qquad\quad\ \,
    \BB_{\ell}
    \equiv
    \set{t_\ell \le T_+(a_\ell;a_{\ell-1},t_{\ell-1})}
    ,\quad
    \\\nonumber &
    \CC_{\ell}
    \equiv
    \begin{cases}
    \set{t_\ell > T_+(a_{\ell}; a_{\ell}, 0)}
    \cap
    (\cap_{r=1}^{\ell-2}
    \set{t_\ell > T_+(a_\ell;a_{r},t_{r})})
    & \ell \ge 2\\
    \varnothing^c
    & \ell = 1\\
    \end{cases}
    ,\\\label{e:minevents}
    &
    \DD^1_{\ell}
    \equiv
    \set{Y_{t_{\ell}}
        ((\pd a_{\ell} \cap \pd a^c_{\ell-1})
        \setminus \set{v_{\ell}})
        = \vec 1}
    ,\quad
    \DD^2_{\ell}
    \equiv
    \set{Y_{t_{\ell}}
        ((\pd a_{\ell} \cap \pd a_{\ell-1}
        )\setminus \set{v_{\ell}})
        = \vec 1}
    ,\\\nonumber
    &
    \EE_{\ell}
    \equiv
    \begin{cases}
        \set{Y_{t_{\ell}}(\pd a_{\ell-1}\setminus\pd a_{\ell}) \neq \vec 1}
        &
        v_\ell\in\pd a_{\ell-1}\cap \pd a_\ell, a_{\ell-1} \neq a_{\ell}
    \\
    \varnothing^c
    &\textup{otherwise}
    \end{cases}
    .
\end{align}
 We say that $\gamma_L$ is a \emph{minimal path} if  $(v_\ell,a_\ell,t_\ell)$ is a minimal step, for each $1\le \ell\le L$.
We further say that $\gamma_L$ is a \emph{minimal path up to time $t$}, if it is a minimal path and $t_L<t\le T_+(v_L;t_L)$.
Let $\Gamma_{\min,L} $ be the set of minimal paths of length $L$. Similarly, we define $\Gamma_{\min,L}(t) \equiv \{\gamma_L \in  \Gamma_{\min,L} : \gamma_L \text{ is a minimal path up to time }t  \} $. We note that both  $\Gamma_{\min,L} $ and $\Gamma_{\min,L}(t)  $ depend on the update sequence $\xi$.
\end{dfn}
\begin{rmk}\bpp
The six events above can be roughly explained as the following:
\begin{enumerate}
\item $\A_\ell$: There is an update to $1$ at vertex $v_\ell$ at time $t_\ell$.
\item $\BB_\ell$: At time $t_\ell$, $a_\ell$ has not been deactivated from the step \emph{immedaitely before} it (i.e., the activation of $a_{\ell-1}$ at time $t_{\ell-1}$).
\item $\CC_\ell$: At time $t_\ell$, $a_\ell$ has been deactivated from all steps \emph{at least two steps ago} (i.e., the activations of $a_{r}$ at time $t_{r}$ for $r\le \ell-2$).
\item $\DD^1_\ell\cap \DD^2_\ell$: At time $t_\ell$, the configuration on $\pd a_\ell \setminus\set{v_\ell}$ is all $1$. Thus update at $v_\ell$ is prone to be blocked by $a_\ell$.
Here we differentiate $\DD^1_\ell$ and $\DD^2_\ell$ because the conditional distribution of $Y_{t_\ell}$ on $\pd a_\ell \cap a_{\ell-1}$ and $\pd a_\ell \cap a_{\ell-1}^c$ are very different and are easier to be analysed separately.
\item $\EE_\ell$: If consider the $\ell$'th step being $(v_\ell, a_{\ell-1}, t_\ell)$ instead of $(v_\ell,a_\ell,t_\ell)$, then the new tuple $(v_\ell, a_{\ell-1}, t_\ell)$ does not satisfy the first five events  because it violates $\DD^2_\ell$. This requirement ensures that we stay at the same hyperedge whenever possible.
\end{enumerate}
In the definition above, event $\A_\ell\cap\DD^1_\ell\cap\DD^2_\ell$  guarrantees that $(v_\ell,a_\ell)$ is activated at time $t_\ell$.
The event $\BB_\ell$ guarrantees that $a_{\ell}$ has not been deactivated w.r.t.~$a_{\ell-1}$ after time $t_{\ell-1}$. 
Together, events $\set{\A_\ell,\BB_\ell,\DD^1_\ell,\DD^2_\ell}_{1\le\ell\le L}$ imply that $\gamma_L$ can potentially be the projection of some discrepancy sequence $\zeta$ (as will be proved in the lemma below).  

Further, the events $\set{\CC_\ell}_{1\le\ell\le L}$ imply that no subpath of $\gamma_L$ satisfies the same condition while $\set{\EE_\ell}_{1\le\ell\le L}$ further require paths to stay at the same hyperedge whenever possible (this requirement is imposed to obtain a better control on the number of minimal paths), justifying the name ``minimal".
\end{rmk}
\begin{lem}\label{l:reduction}
    For each update sequence $\xi$, if $\tcoup>T$, then there exists a constant $L\ge 0$ and a minimal path $\gamma_L = ((v_\ell, a_\ell, t_\ell))_{0\le \ell \le L} \in \Gamma_{\min,L}(T)$.
\end{lem}
\begin{proof}
    Recall the construction of a discrepancy sequence $\zeta = \zeta(\xi) = ((u_i,b_i,s_i))_{0\le i \le M}$ (here we use the representation moving forward in time, with $(b_0,s_0)=(b_1,0)$). It satisfies  for all $1 \le i \le M$, that (a) $u_{i-1},u_i\in \pd b_i$, (b) $(u_i,s_i,1)\in\xi$, (c) $Y_{s_i}(b_i) =\vec 1$ and (d)
\begin{equation}\label{e:before_deact}
    s_{i-1} < s_{i} \le T_+(u_{i-1};s_{i-1}),
    \quad \text{hence in particular,} \quad
    s_{i-1} < s_{i}\le T_+(b_i;b_{i-1},s_{i-1}).
\end{equation}
One can construct a minimal path based on $\zeta$ as follows:
    \begin{enumerate}[\quad 1.]
        \item Let $\gamma^1 \equiv ((v^1_\ell,a^1_\ell,t^1_\ell))_{0\le\ell\le 1} \equiv ((u_\ell,b_\ell,s_\ell))_{0\le\ell\le 1}$. It follows from the construction of a discrepancy sequence, in particular from the fact that $b_0 = b_1$, that $\gamma^1$ is a minimal path of length $L_1\equiv 1$  up to time $s_2$.
        \item For $2\le i\le M$, suppose we have already constructed $\gamma^{i-1} = ((v^{i-1}_\ell,a^{i-1}_\ell,t^{i-1}_\ell))_{0\le \ell \le L_{i-1}}\in \Gamma_{\min, L_{i-1}} (s_i)$ with $L_{i-1} \le i-1$ and
\begin{equation}\label{e:minpath}
    (v^{i-1}_{L_{i-1}},
    t^{i-1}_{L_{i-1}})
    = (u_{i-1},
    t_{i-1})
    .
\end{equation}
To construct $\gamma^{i}$, let
\[
    \ell_\star \equiv
    \ell_\star(i)
    \equiv
    \begin{cases}
    0
    &s_i \le T_+(b_i;b_i,0)\\
    \min_{1\le \ell \le L_{i-1}}
    \set{\ell: s_i \le T_+(b_i;a^{i-1}_\ell,t^{i-1}_\ell)}
    & \textup{otherwise}
    \end{cases}
    .
\]
Note that $\ell_\star$ is well-defined since by \eqref{e:before_deact} and \eqref{e:minpath}, the condition $\set{\ell: s_i \le T_+(b_i;a^{i-1}_\ell,t^{i-1}_\ell)}$ is satisfied by $\ell = L_{i-1}$. If $\ell_\star = 0$, then there exists
$v_\star \in \pd b_i$
such that $s_i \le T_+(v_\star; 0)$. In this case, we define $L_i = 1$ and
$$\gamma^i \equiv
((v^i_\ell,a^i_\ell,t^i_\ell))_{0\le\ell\le 1}
\equiv
((v_\star, b_i, 0), (u_i,b_i,s_i)).
$$
 Otherwise, we define $L_i = \ell_\star + 1$, $(v^i_\ell, a^i_\ell, t^i_\ell) \equiv  (v^{i-1}_\ell,a^{i-1}_\ell,t^{i-1}_\ell), 0\le \ell \le \ell_\star$ and
\begin{equation}\label{e:EEreason}
    (v^i_{L_i},a^i_{L_i},t^i_{L_i}) \equiv
    \begin{cases}
        (u_i,a^{i-1}_{\ell_\star},s_i)
        & \textup{if } u_i \in \pd a^{i-1}_{\ell_\star}
        \textup{ and $Y_{s_i}(\pd a^{i-1}_{\ell_\star})=\vec{1} $,}\\
        (u_i,b_i,s_i) & \textup{otherwise.}
    \end{cases}
\end{equation}
In either case, one can check that the six events defined in \eqref{e:minevents} are satisfied for $\ell = L_i$ and hence $(v^i_{L_i},a^i_{L_i},t^i_{L_i})$  is a minimal step of $\gamma^i$. Indeed, the occurrence of $\A_{L_i}, \BB_{L_i}, \DD^1_{L_i},\DD^2_{L_i}$ follows from the construction of a discrepancy sequence, the occurrence of $\CC_{L_i}$ follows from the minimality of $\ell_\star$ and that of $\EE_{L_i}$ follows from \eqref{e:EEreason}. By construction, $((v^i_r,a^i_r,t^i_r))_{0\le\ell\le L_i-1}$ is a subpath of $\gamma^{i-1}\in\Gamma_{\min,L_{i-1}}(s_i)$ and hence is minimal path itself. Therefore $\gamma^i$ is a minimal path of length $L_i$ up to $s_{i+1}$.
\end{enumerate}
To conclude the proof, one can take $\gamma^M$ and note that $T_+(v_{L_M}^M;t_{L_M}^M) = T_+(u_M;s_M)\ge T$ by the definition of $\zeta$.
\end{proof}
\begin{rmk}\label{r:proj}
    We will use $\Gamma_{\proj, L}\subset \Gamma_{\min,L}$ to denote the set of minimal paths that are projected from some discrepancy sequences as in the proof above.
\end{rmk}
Lemma~\ref{l:reduction} implies that for every time $T\ge 0$ and integer $L\ge 1$,
\begin{align}
    \P(\tcoup\ge T) &\le
    \P( \Gamma_{\proj,L-1}  (T)\neq \varnothing )
    + \P(\Gamma_{\proj,L} \neq \varnothing)
    \label{e:twocases_proj}
    \\&\le
    \P( \Gamma_{\proj,L-1}  (T)\neq \varnothing )
    + \P(\Gamma_{\min,L} \neq \varnothing)
    .\label{e:twocases}
\end{align}

The next lemma bounds the first term on the right hand side.
\begin{lem}\label{l:longjumps}
    Let $T \in \N$. Denote  $L=\lfloor cT \rfloor $, where $c= \frac{1}{4 \log (2k^{2} \Delta)}$. Then
\[
    \P( \Gamma_{\proj,L-1}  (T)\neq \varnothing )
        \le n\exp(- T/4). \]
\end{lem}
\begin{proof}
Consider constructing a path in $\Gamma_{\proj,L-1}  (T)$ by first choosing the locations of ``jumps" and then picking their times. The number of ways of choosing a sequence $\vec{w} \equiv (a_\ell)_{0\le\ell\le L-1}$ such that  $\pd a_{\ell}\cap \pd a_{\ell-1}\neq\varnothing$  for all $1\le \ell\le L-1$, is at most $n (\Delta k)^{L}$.
For each fixed $\vec w$, recursively define $T_0 \equiv 0$ and $T_{\ell+1} \equiv T_+(a_{\ell+1};a_{\ell  }, T_{\ell})$ for $0\le \ell\le L-2$. If there exists a sequences of times $\vec s \equiv (s_\ell)_{0\le \ell\le L-1}$ and vertices $\vec v \equiv(v_\ell)_{0 \le \ell \le L-1}$ such that the path $\gamma_{L-1} \equiv ((v_\ell,a_\ell,s_\ell))_{0\le\ell\le L-1}$ is a minimal path up to time $T$, then $s_0=T_0=0$, and one can  show inductively that
   \[
       s_\ell \le T_+(a_{\ell};a_{\ell - 1},s_{\ell-1})
   \le T_+(a_{\ell};a_{\ell - 1},T_{\ell-1}) = T_\ell,
   \]
   using the induction step ($s_{\ell-1} \le T_{\ell-1}$) and the monotonicity of  $T_+(a;b,t)$ in $t$. The existence of $\vec s$ further implies that
\[
    T_{L}\equiv  T_+(a_{L-1};a_{L-1}, T_{L-1})
    \ge T_+(v_{L-1};s_{L-1}) \ge T.
\]

   By construction, the joint law of $(T_{\ell} - T_{\ell-1})_{0 \le \ell \le L}$ is stochastically dominated by that of $(Z_{\ell})_{0 \le \ell \le L}$ where $Z_{\ell}:=\max(Z_{\ell,1},\ldots,Z_{\ell,k})$ and $(Z_{\ell,i})_{1 \le i \le k,1 \le \ell \le L}$ are i.i.d.~exponential random variables with rate $1$. Note that using the order statistic of $Z_{\ell,1},\ldots,Z_{\ell,k} $, we can decompose $Z_{\ell}$ into a sum $\sum_{i=1}^{k}J_{\ell,i}$ of independent exponential r.v.'s with $\mathbb{E}[J_{\ell,i}]=i^{-1}$. Hence for all $\lambda \in (0,1/2) $ and $\ell \le L $,
\[
\mathbb{E}[e^{\lambda Z_{\ell}}]
= \prod_{i=1}^k \mathbb{E}[e^{\lambda J_{\ell,i}}]
=  \prod_{i=1}^k
\Big(1+\frac{\lambda}{i-\lambda}\Big)
\le \prod_{i=1}^k \Big(1+\frac{1}{i}\Big)
\le 2e^{\frac{1}{2}+\cdots+\frac{1}{k}}
\le e^{\log (2k)}
.
\]
   By Markov's inequality, independence of $Z_{\ell}$'s and the aforementioned stochastic domination,
\[
    \P(T_{L}\ge T)
    \le \mathbb{E}[
    e^{(T_L-T)/2}]
    =
    e^{-T/2} e^{L\log (2k)}
    \le e^{-[1-2c \log (2k) ]T/2}.
\]
Therefore by the arguments above together with the choice $c= \frac{1}{4 \log (2k^{2} \Delta)}$,
\[
    \P( \Gamma_{\proj,L-1}  (T)\neq \varnothing )
    \le n(\Delta k)^{cT} \P(T_L\ge T)
    \le n(\Delta k)^{cT} e^{-[1-2c \log (2k) ]T/2} = n e^{-T/4}
    ,
\]
as desired.
\end{proof}

\subsection{Redacted path}
 In the \bpp remainder \ehl of the section, we bound the size of $\Gamma_{\min,L}$. Our basic approach is to bound for each $\gamma_L\in\Gamma_{\min,L}$ the expected number of ways to extend $\gamma_L$ by two steps. However, the ``vanilla-version" of this expectation can be much bigger than 1. Intuitively, for general hypergraphs, hyperedges may be highly overlapping with each other. Thus when one hyperedge is all $1$ in process $Y_t$, its neighbouring hyperedges may also be all $1$ with little extra cost. This phenomenon leads to numerous local ``tangles'' where multiple choices of the immediate next step exist for the same second-next step, blowing up the number of extensions significantly.

To overcome the aforementioned obstacle, we divide the steps of a minimal path into good branchings (i.e.,\ small overlaps) and bad branchings (i.e.,\ large overlaps) and skip the ``tangles'' by ignoring the first step of a bad branching and recording only the ``key'' step instead.
More precisely, given a path $\gamma_L=\set{(u_\ell, b_\ell,s_\ell)}_{0\le \ell \le L}\in\Gamma_L$, we classify each of the steps $1\le \ell \le L$ into one of the following three cases: (see also Figure~\ref{fig:threebranchings})
\begin{figure}[t]
\begin{center}
    \includegraphics{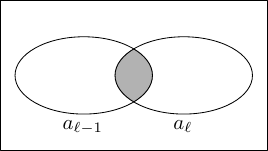}
    \includegraphics{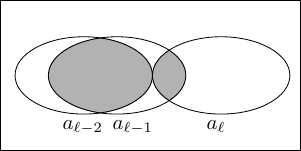}
    \includegraphics{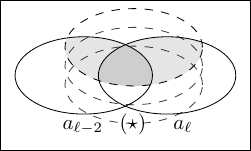}
    \caption{Good/Type-I/Type-II branching in redacted paths}\label{fig:threebranchings}
\end{center}
\end{figure}
\begin{enumerate}
    \item We say that $(v_\ell,a_\ell,t_\ell)$ is a good branching if $|\pd a_{\ell-1} \cap \pd a_{\ell}| \le k/3$.
    \item We say that $\set{(v_{\ell-1},a_{\ell-1},t_{\ell-1}), (v_\ell,a_\ell,t_\ell)}$ is a type-I (bad) branching if
\[
|\pd a_{\ell-2} \cap \pd a_{\ell-1}| > k/3
, \quad
\pd a_{\ell-2} \cap \pd a_{\ell} = \varnothing
.
\]
\item We say that $\set{(v_{\ell-1},a_{\ell-1},t_{\ell-1}), (v_\ell,a_\ell,t_\ell)}$ is a type-II (bad) branching if
\[
|\pd a_{\ell-2} \cap \pd a_{\ell-1}| > k/3
, \quad
\pd a_{\ell-2} \cap \pd a_{\ell} \ne \varnothing
.
\]
\end{enumerate}
For each $\gamma_L=((v_\ell, a_\ell,t_\ell))_{0\le \ell \le L}\in\Gamma_L$, we use the following greedy algorithm to partition $\gamma_L$ into disjoint segments of size one or two such that all but (possibly) the last block satisfy one of the above three types of branchings (we assume $L>1$ to avoid triviality).
\begin{enumerate}[\quad 1.]
\item
    Observe that (by Definition \ref{def:path}) $a_1 = a_0$ implies that $\pd a_2 \cap \pd a_0\neq\varnothing$. Therefore $\set{(v_1,a_1,t_1),(v_2,a_2,t_2)}$ forms a type-II branching and is taken as a block in the partition.
\item
For each $\ell\le L-2$,
if after a certain number of iterations we have partitioned the first $\ell$ steps of $\gamma_L$ into one of the above three cases, then we look at the $(\ell+1)$'th step $(v_{\ell+1}, a_{\ell+1},t_{\ell+1})$:
\begin{itemize}
    \item If $(v_{\ell+1}, a_{\ell+1},t_{\ell+1})$ forms a good branching then we take $\set{(v_{\ell+1}, a_{\ell+1},t_{\ell+1})}$ as a block of size 1 in the partition.
    \item
Otherwise, we take $\{(v_{\ell+1}, a_{\ell+1},t_{\ell+1}),(v_{\ell+2}, a_{\ell+2},t_{\ell+2})\} $ as a segment of size~2 in the partition. By definition, it must either form a type-I branching or a type-II branching.
\end{itemize}
\item If step 2  ends with a complete partition of the path (i.e.,~the last two steps form a type-I or type-II branching), then we are done. Otherwise, it must give a partition of the first $L-1$ steps of $\gamma_L$. To  obtain a partition of the first $L$ steps we then let $(v_L,a_L,t_L)$ form a block of size 1, and call it a \emph{half-branching} if it is not a good-branching itself.
\end{enumerate}
Given the unique partition defined above, we define $\Ig, I_\textup{1}, I_\star, I_\textup{2} \subseteq \set{1,\dots, L}\equiv [L]$, each as a function of $\gamma_L\in\Gamma_L$, as the set of (the indices of) steps that belong to good branchings, type-I branchings, the first step of type-II branchings and the second step of type-II branchings, respectively, in the partition of $\gamma_L$ by the above procedure. The four sets $\Ig, I_\textup{1}, I_\star, I_\textup{2}$ form a partition of $[L]$ or $[L-1]$. In particular, $1 \in I_{\star}$, $2 \in I_2$ and $L$ belongs to none of the four sets if the last step is a half-branching.

\begin{dfn}
    Given a sequence $\tgm_L\equiv \set{(v_\ell, a_\ell,t_\ell)}_{0\le \ell \le L}$, we say that $\tgm_L$ is a \emph{redacted path} in $\tG$ of length $L$ if there exists path $\gamma_L=\set{(u_\ell, b_\ell,s_\ell)}_{0\le \ell \le L}\in\Gamma_L$ such that for each $0\le \ell\le L$,
\begin{equation}
    (v_{\ell},a_{\ell},t_{\ell})
    =
    \begin{cases}
    (u_{\ell}, b_{\ell}, t_{\ell}),
    &
    \ell \in \set{0,L} \cup \Ig(\gamma_L) \cup I_1(\gamma_L) \cup I_2(\gamma_L)
    \\
    (\star,\star,\star),
    &
    \ell \in I_\star(\gamma_L)
    \end{cases}
    .
    \label{e:gmtotgm}
\end{equation}
We will refer to steps equaling to $(\star,\star,\star)$ as \emph{redacted steps}.

With some abuse of notation, we define $\Ig\equiv \Ig(\tgm_L)\equiv \Ig(\gamma_L)$ and define $I_1,I_2,I_\star$ similarly. 
Here we remark that for redacted path $\tgm_L$, the set $I_\star$ (resp.~$I_2$) can be defined as the collection of redacted steps (resp.~the collection of steps following a redacted step) and thus does not depend on the choice of $\gamma_L$.
Let $  M(\tgm_L) \equiv  |\Ig|+\frac{1}{2} |I_1| +|I_2|$ be the \emph{branching length} of $\tgm_L$, where the factor $\frac{1}{2}$ is taken because we want each type-I branching to contribute $+1$ to $M(\tgm_L)$.
We define
\[
    \begin{array}{ll}
        \tGm_M \hspace{-6pt}&\equiv
    \set{\tgm_L\in \tGm:
    M(\tgm_L) = M, \quad L\in \Ig \cup I_1\cup I_2}
    ,
    \\
    \tGm_{M+1/2} \hspace{-6pt}&\equiv
    \set{\tgm_L\in \tGm:
    M(\tgm_L) = M, \quad L\notin \Ig \cup I_1\cup I_2}
    ,
\end{array}
\]
to be the set of redacted paths of branching length $M$ that end with regular branchings or half-branchings respectively.
We finally denote the set of all redacted paths by $\tGm$.
\end{dfn}

In the definition above, redacted steps, i.e.,\ steps equaling to $(\star,\star,\star)$, can never appear twice in a row, and the  last step of $\tgm_L$ is never redacted.
Let
\[
    \Im \equiv \Im(\tgm_L)
    \equiv \set{\ell: \ell,\ell-1\notin I_\star(\tgm_L)}
    =\begin{cases}
        \Ig \cup I_1 & L\in I_2\\
        \Ig \cup I_1 \cup \set{L} & \textup{otherwise}
    \end{cases}
    .
\]
Using the convention that  $T_+(v;\star,\star) = T_+(a;\star,\star) = 0$ for all $v\in V$ and $a\in F$, we can extend the definition of a minimal step to redacted paths.
\begin{dfn}
We say that $(v_\ell,a_\ell,t_\ell)$ is a \emph{minimal step} in $\tgm_L= ((v_i, a_i,t_i))_{0\le i \le L}$ if and only if $\ell \in \Im$ and $(v_\ell,a_\ell,t_\ell)$ satisfies the six events of \eqref{e:minevents}.
For  $\ell\in I_2$ we say that $(v_{\ell},a_{\ell}, t_{\ell})$ is a \emph{minimal type-II step}, if all of the following events hold:
\begin{equation}
    \begin{array}{l}
    \A_{\ell}
    \equiv
    \set{(v_{\ell},t_{\ell},1)\in \xi}
    ,\quad
    \widetilde{\BB}_{\ell}
    \equiv
    \set{t_{\ell} \le
        T_+(a_{\ell};T_+(a_{\ell-2};t_{\ell-2}))}
    ,\smallskip\\
    \CC_{\ell}
    \equiv
    \set{t_\ell > T_+(a_{\ell}; a_{\ell}, 0)}
    \cap
    (\cap_{r=1}^{\ell-2}
    \set{t_\ell > T_+(a_\ell;a_{r},t_{r})})
    ,
    \smallskip\\
    \widetilde{\DD}_{\ell} \equiv
    \set{Y_{t_{\ell}}
        (\pd a_{\ell} \setminus \set{v_{\ell}})
        = \vec 1}
    .
\end{array}
\label{e:typeIIevents}
\end{equation}
\end{dfn}
\begin{dfn}
    Fix an update sequence $\xi$ and a redacted path $\tgm_{L}= ((v_\ell, a_\ell,t_\ell))_{0\le \ell \le L}\in \tGm$. We say that $\tgm_L\in\tGm$ is a \emph{minimal redacted path} if,
\begin{enumerate}
    \item For each $\ell\in \Im$, $(v_{\ell},a_{\ell}, t_{\ell})$ is a minimal step of $\tgm_L$.
    \item For each $\ell\in I_2$, $(v_{\ell},a_{\ell}, t_{\ell})$ is a minimal type-II step of $\tgm_L$.
    \end{enumerate}
    We denote the set of minimal redacted paths by $\tGm_{\min}$ and define $\tGm_{\min,M}\equiv\{\tilde \gamma \in\tGm_{\min}:\tgm \text{ has branching length }M  \}$ where $M$ takes value from  $\frac{1}{2}\N$.
\end{dfn}
\begin{lem} \label{l:mintored}
    For every update sequence $\xi$ and length $L\ge 0$, if $\Gamma_{\min,L}\neq \varnothing$, then $\tGm_{\min,M}\cup \tGm_{\min,M+1/2} \neq \varnothing$ where $M= \lfloor (L-1)/2\rfloor$.
\end{lem}
\begin{proof}
For each $\gamma_L=((u_\ell,b_\ell,s_\ell))_{0\le \ell \le L}\in \Gamma_{\min,L}$, we can construct $\tgm_L=((v_\ell,a_\ell,t_\ell))_{0\le\ell\le L}$ according to \eqref{e:gmtotgm}. It is straight forward to check that $\tgm_L$ is a redacted path with branching length $M(\tgm_L) \ge (L-1)/2$ and for all $\ell\in \Im$ we have that $(v_\ell,a_\ell,t_\ell)$ is a minimal step of $\tgm_L$. We now verify that for all $\ell\in I_2$ we have that $(v_\ell,a_\ell,t_\ell)$ is a minimal type-II step of $\tgm_L$. To do so, it suffices to verify that if $\ell \in I_2$ then $\widetilde{\BB}_\ell$ holds (the occurrence of $\A_{\ell},\BB_\ell$ and $\widetilde{\DD}_{\ell} $ follows from the fact that $\gamma_L \in \Gamma_{\min,L}  $). Fix some $\gamma_L\in \Gamma_{\min, L}$ and $\ell \in I_2$.  From \eqref{e:minevents} we know that events $\BB_{\ell-1}$ and $\BB_{\ell}$ hold. It follows that
\begin{align*}
    \begin{array}{lll}
    t_{\ell-1}
    \hspace{-6pt}&\le T_+(a_{\ell-1};a_{\ell-2},t_{\ell-2})
    \hspace{-6pt}&\le T_+(a_{\ell-2};t_{\ell-2})
    ,
    \\t_{\ell}
    \hspace{-6pt}&\le T_+(a_{\ell};a_{\ell-1},t_{\ell-1})
    \hspace{-6pt}&\le T_+(a_{\ell};t_{\ell-1})
    .
    \end{array}
\end{align*}
Combining the last two equations and using the monotonicity of $T_+(a;t)$ in $t$ yields that
\[
    t_{\ell}\le T_+(a_{\ell};T_+(a_{\ell-2};t_{\ell-2})).
\]
Consequently, $\tgm_L\in \tGm_{\min}$. Truncating $\tgm_L$ such that $M= \lfloor (L-1)/2\rfloor$ concludes the proof.
\end{proof}

\subsection{Recursion of two steps}
In this subsection we complete the recursion on $\tGm_{\min,M}$ and conclude the proof of Theorem \ref{t:mixing}. The main idea is to bound the expected number  of ways of extending a redacted path by each one of the three types of branchings. Since every step of a good branching or a type-I branching is also a minimal step, we split the discussion according to minimal steps and minimal type-II steps.
For each $a\in F$, let $\bN(a) \equiv \set{b\in F: \pd a \cap \pd b \neq \varnothing}$ be the hyperedge-neighbourhood of $a$.
Define
\[
     \bN^> (a)
     \equiv \set{b\in \bN(a):
     |\pd a \cap \pd b| > k/3}
     ,\quad
     \bN^\le (a) \equiv \bN(a) \setminus \bN^> (a)
     .
 \]
 Throughout the section, we will use $\tgm_L$ to denote the redacted paths of length $L$ (but varying branching length).

Fix a path $\tgm_L=((v_{\ell},a_{\ell},t_{\ell}) )_{0 \le \ell \le L }\in \tGm$ and a vertex-hyperedge pair $(v,a)$ satisfying $a\in \bN(a_L)$, $v\in \pd a$. We denote by $\tgm^\tm_{L+1}(v,a,t)=((v_{\ell},a_{\ell},t_{\ell}) )_{0 \le \ell \le L+1 }$  the redacted path extended from $\tgm_L$ with $(v_{L+1},a_{L+1},t_{L+1}) \equiv (v,a,t)$. We define
\begin{equation}
\label{e:NtmL}
N^\tm_{L} \equiv
N^\tm(\tgm_{L};v,a)
    \equiv |\set{t:\tgm^\tm_{L+1}(v,a,t) \in \tGm_{\min},\,
    L+1 \in \Im (\tgm^\tm_{L+1}(v,a,t))}|
\end{equation}
to count the number of possible minimal steps using $(v,a)$.
Note that $N^\tm_{L}$ is a.s.\ finite since it is  bounded from above by the number of updates at $v$ between time $t_{L}$ and $T_+(a_{L+1};a_{L},t_{L})$, which in turn is a.s.\ finite.
We further use $\tgm_{L+2}^\textup{II}(v,a,t)=((v_{\ell},a_{\ell},t_{\ell}) )_{0 \le \ell \le L+2 }$ to denote the redacted path extended from $\tgm_L$ with
\[(v_{L+1},a_{L+1},t_{L+1}) \equiv (\star,\star,\star),\quad (v_{L+2},a_{L+2},t_{L+2}) \equiv (v,a,t)\]
and define
\[
    N^\tII_{L}\equiv
    N^\tII(\tgm_{L};v,a)
    \equiv |\set{t:\gamma^\tII_{L+2}(v,a,t) \in \tGm_{\min,M+1}
        ,\  L+2 \in I_2(\tgm_{L+2}^\tII(v,a,t))
    }|
.
\]
Finally, for each integer $M\ge M(\tgm_L)$, we let $\tilde \Gamma_{\min,M}(\tgm_L)$ be the collection of redacted paths in $\tGm_{\min,M}$ that agree with $\tgm_L$ in their first $L$ steps and write $N_{\min,M}(\tgm_L):=|\tilde \Gamma_{\min,M}(\tgm_L)| $.
\begin{lem}\label{l:Eminbranch}
For every integer $L \ge 1$, $\tgm_L\in\tGm$,  $a\in \bN(a_L)$ and $v\in \pd a$,
\[
    \E[ N^\tm(\tgm_L;v,a) \mid \tgm_L\in\tGm_{\min}]
    \le {C_\tm}
    \begin{cases}
        \frac{1}{k} & a = a_{L}
        , \\
        \frac{1}{m_{L+1}+1} 2^{-(k-1-m_{L+1})}
        & a \neq a_{L}, v \notin a_{L}
        , \\
        \frac{k-m_{L+1}}{k}\frac{1}{m_{L+1}+1 }
        2^{-(k-1-m_{L+1} )}
        & a \neq a_{L}, v \in a_{L}
        ,
    \end{cases}
\]
where $m_{L+1} \equiv |\pd a \cap \pd a_L\setminus\set{v}|$ and $C_\tm$ is an absolute constant independent of $\Delta,k$.
\end{lem}

\begin{lem}\label{l:EtypeII}
For every integers $L\ge M\ge 1$, $\tgm_L\in\tGm_M$,  $a\in \bN(a_L)$ and $v\in \pd a$,
\[
    \E[ N^\tII(\tgm_L;v,a) \mid \tgm_L\in\tGm_{\min,M}]
    \le  2^{-(k-1)}(1+\log k).
\]
\end{lem}

The proof of Lemma~\ref{l:Eminbranch} and Lemma~\ref{l:EtypeII} is postponed to Section~\ref{s:minbranch} and Section~\ref{s:typeIIbranch}, respectively.
We first apply both lemmas to derive the main result
\begin{thm}\label{t:Etwosteps}
For all integers $L\ge M\ge 1$,  and redacted path $\tgm_{L} \in\tGm_M$,
\[
    \E[N_{\min,M+1}(\tgm_L)\mid \tgm_{L}\in\tGm_{\min,M}]
    \le C
        \begin{cases}
            \Delta k^2 2^{-k}
            & G \textup{ is linear} \\
            \Delta^2 2^{-k} + \Delta k^22^{-2k/3}
            &\textup{otherwise}
        \end{cases},
\]
where $C$ is an absolute constant independent of $\Delta,k$.
\end{thm}
\begin{proof}
Fix some $\tgm_L=((v_\ell, a_\ell,t_\ell))_{0\le \ell \le L}\in\tGm_M$. For brevity define $\tE_L\equiv\E[\cdot\mid \tgm_L\in\tGm_{\min,M}] = \E[\cdot\mid \tgm_L\in\tGm_{\min}]$.
Let $N^\tg_{\min,M+1}(\tgm_L)$ be the number of redacted paths in $\tGm_{\min,M+1}(\gamma_L)$ such that their last block is a good branching and define $N^\tI_{\min,M+1}(\tgm_L)$ and $N^\tII_{\min,M+1}(\tgm_L)$ similarly.
By the construction of redacted paths,
\[
    N_{\min,M+1}(\tgm_L) =
    N^\tg_{\min,M+1}(\tgm_L)
    + N^\tI_{\min,M+1}(\tgm_L)
    + N^\tII_{\min,M+1}(\tgm_L)
    .
\]
We bound the three cases separately:
\begin{enumerate}[1.]
\item
We first bound $\tE_{L}[N^\tg_{\min,M+1}(\tgm_{L})]$. For any two hyperedges $a,b\in F$, let $m(a,b) \equiv |\pd a \cap \pd b|$ be the size of their overlap.
For each $a\in N^\le(a_L)$, applying Lemma~\ref{l:Eminbranch} to $v\in \pd a\setminus \pd a_L$ and $v\in \pd a \cap \pd a_{L}$ separately yields that
\begin{align*}
    \sum_{ v\in\pd a \setminus \pd a_{L}}
    \tE_L N^\tm(\tgm_{L};v,a)
    &\le
    C_\tm
    \cdot
    \frac{|\pd a\setminus \pd a_{L}|}{m(a,a_{L})+1}
    2^{-(k-1-m(a,a_{L}))}
    ,\\
    \sum_{ v\in\pd a \cap \pd a_{L}}
    \tE_L N^\tm(\tgm_{L};v,a)
    &\le
    C_\tm
    \cdot
    |\pd a\cap \pd a_{L}|
    \cdot
    \frac{k+1-m(a,a_{L})}{km(a,a_{L})}
        2^{-(k-m(a,a_{L}))}
    .
\end{align*}
Combining the two estimates yields the following upper bound on the number of good branchings,
\begin{align}
    \tE_{L}[N^\tg_{\min,M+1}(\tgm_{L})]
    &= \sum_{a\in \bN^\le(a_L)}
    \sum_{v\in \pd a}
    \tE_{L}
    [N^\tm (\tgm_{L};v,a)]
    \nonumber\\&\le
    3C_\tm
    \sum_{a\in \bN^\le(a_L)}
    \frac{k+1-m(a,a_{L})}{m(a,a_{L})}
        2^{-(k-m(a,a_{L}))}
    \le
    3C_\tm
    (\Delta k) \cdot
    k 2^{-2k/3}
    .
    \label{e:ENgtgm}
\end{align}
\item
    We now bound the expected number of type-I branchings. If $\tgm_{L+2}=((v_\ell, a_\ell,t_\ell))_{0\le \ell \le L+2}\in \tGm_{\min,M+1}$ is extended from $\tgm_L$ via a type-I branching, then $\tgm_{L+1}$, defined as the first $(L+1)$ steps of $\tgm_{L+2}$, must have branching length  $M+\frac{1}{2}$.
    We first enumerate the possible extensions of $\tgm_L$ (via a type-I branching) through some fixed redacted path $\tgm_{L+1} \equiv \tgm_{L+1}^\tm(v_{L+1},a_{L+1},t_{L+1})$.
Let
\[
    A_{L+1} \equiv A_{L+1}(\tgm_{L+1})
    \equiv \bN(a_{L+1}) \setminus \bN(a_L)
\]
be the set of $a_{L+2}$'s that together with $a_{L+1}$ may form a type-I branching extended from $\tgm_L$.
For each $a\in A_{L+1}$, we split the discussion into two cases, $v\in \pd a\setminus \pd a_{L+1}$ and $v\in \pd a \cap \pd a_{L+1}$. Define $\tE_{L+1}\equiv\E[\cdot\mid \tgm_{L+1}\in\tGm_{\min, M+1/2}]$. Applying similar reasonings as in the derivation of \eqref{e:ENgtgm}, we get that
\begin{align}
    \tE_{L+1}[N^\tI_{\min,M+1}(\tgm_{L+1})]
    &\le
    3C_\tm
    \sum_{a\in A_{L+1}}
    \frac{k+1-m(a,a_{L+1})}{m(a,a_{L+1})}
       2^{-(k-m(a,a_{L+1}))}
    \label{e:EgammaL1}
    \\&\equiv
    3C_\tm
    \sum_{a\in A_{L+1}}
    m(a,a_{L+1}) f(m(a,a_{L+1}))
    \nonumber,
\end{align}
where
\[
   f(m) \equiv \frac{k+1-m}{m^2}2^{-(k-m)}.
\]
The function $f(m)$ is an increasing function on $3\le m \le k-1$ and $f(1),f(2)\le 3f(3)$. Recall that $\pd a_{L+2}\cap \pd a_L = \varnothing$. Therefore the sizes of overlaps $\set{m(a,a_L)}_{a\in A_{L+1}}$ satisfy that
\[
    m(a,a_{L+1}) 
    \le  |\pd a_{L+1}\setminus \pd a_L|
    ,\
    \sum_{a\in A_{L+1}} m(a,a_{L+1})
    \le \Delta  |\pd a_{L+1}\setminus \pd a_L|
    = \Delta (k-m(a_L,a_{L+1}))
    .
\]
Therefore,
\begin{align*}
    \tE_{L+1}[N^\tI_{\min,M+1}(\tgm_{L+1})]
    &\le 9C_\tm
    f(k-m(a_L,a_{L+1}))
    \sum_{a\in A_{L+1}} m(a,a_{L+1})
    \\&=
    9C_\tm
    \Delta \frac{m(a_L,a_{L+1})+1}{k-m(a_L,a_{L+1})}2^{-m(a_L,a_{L+1})}.
\end{align*}
Now we sum over all possible choices of $(v_{L+1},a_{L+1})$.
Observe that if $a_{L+1}= a_L$, then $\pd a_{L}\cap \pd a_{L+2}\neq \varnothing$, leading to a type-II branching.
Thus $A_{L}\equiv \bN^>(a_L)\setminus\set{a_L}$ is the set of possible $a_{L+1}$'s.
Observe from Lemma~\ref{l:Eminbranch} that the upper bound of $ N^\tm(\tgm_L;v,a)$ does not depend on $t_{L+1}$.
A similar calculation to \eqref{e:ENgtgm} gives that
\begin{align*}
    \tE_L[N^\tI_{\min,M+1}(\tgm_L)]
    &=
    \sum_{a\in A_{L}}
        \sum_{v\in \pd a}
        \tE_L\big[
        N^\tm(\tgm_L;v,a)
    \tE_{L+1}[N^\tI_{\min,M+1}
        (\tgm^\tm_{L+1}(v,a,t))]
        \big]
    \\&\le
    3{C_\tm} \!\sum_{a\in A_{L}}\!
        \frac{k+1-m(a_L,a)}{m(a_L,a)}
        2^{-(k-m(a_L,a))}
    \cdot
        9{C_\tm}\Delta
        \frac{m(a_L,a)+1}{k-m(a_L,a)}
        2^{-m(a_L,a)}
    \\&\le {108 C^2_\tm}\Delta 2^{-k} |A_{L}|
    \le {324 C^2_\tm} \Delta^2 k2^{-k}
    ,
\end{align*}
where the last step uses the fact that
\[
    |A_L| \le \frac{\Delta k}
    {\min_{a\in \bN^>(a_L)} m(a,a_L)}
    \le 3 \Delta
    .
\]
\item Finally, using Lemma~\ref{l:EtypeII}, we can bound the number of minimal type-II branchings:
\begin{align*}
    \tE_L N^\tII_{\min,M+1}
    &\le \sum_{a\in\bN(a_L)}
    \sum_{v\in \pd a}
    \tE_L N^\tII(\tgm_L;v,a)
    \le (\Delta k^2) \cdot 2^{-(k-1)} (1+\log k
    ).
\end{align*}
 \end{enumerate}
Combining the three cases together completes the proof for general hypergraphs.
\medskip

We now turn to linear hypergraphs. In this setup  all branchings must be good (because $m(a,b)\le 1$ for all $a,b\in F$). Therefore applying \eqref{e:ENgtgm}, we have
\begin{align*}
    \tE_{L}[N_{\min,M+1}(\tgm_{L})]
    &=\tE_{L}[N^\tg_{\min,M+1}(\tgm_{L})]
    = \sum_{a\in \bN(a_L)}
    \sum_{v\in \pd a}
    \tE_{L}
    [N^\tm (\tgm_{L};v,a)]
    \nonumber\\&\le
    3C_\tm
    \sum_{a\in \bN(a_L)}
    \frac{k+1-m(a,a_{L})}{m(a,a_{L})}
        2^{-(k-m(a,a_{L}))}
    =
    6C_\tm
    (\Delta k) \cdot
    k 2^{-k}
    .
\end{align*}
This concludes the proof for linear hypergraphs.
\end{proof}

\begin{proof}[Proof of Theorem~\ref{t:mixing}]
    The assertion of Theorem~\ref{t:mixing} follows  by combining Lemma~\ref{l:reduction}, \eqref{e:twocases}, Lemma~\ref{l:longjumps}, Lemma~\ref{l:mintored} and Theorem~\ref{t:Etwosteps}.
\end{proof}

\subsection{Number of minimal steps}\label{s:minbranch}
In this subsection we prove Lemma~\ref{l:Eminbranch}.
Throughout the section, we assume $\tgm_L\in \tGm, a\in \bN(a_L), v\in \pd a$  and for brevity of notation write $\tgm_{L+1}(t) \equiv \tgm^\tm_{L+1}(v,a,t)$ (recall that  $\tgm^\tm_{L+1}(v,a,t)=((v_{\ell},a_{\ell},t_{\ell}) )_{0 \le \ell \le L+1 }$  is the redacted path extended from $\tgm_L$ with $(v_{L+1},a_{L+1},t_{L+1}) = (v,a,t)$).
Recall the definitions of minimal branching and minimal type-II branching in \eqref{e:minevents} and \eqref{e:typeIIevents}. We  define
\begin{align*}
    \MM^\tm_\ell(t)
    \equiv \MM^\tm_{\ell}(\tgm_{L+1}(t)) &\equiv
    \A_\ell(t) \cap \BB_\ell(t) \cap \CC_\ell(t)
    \cap \DD^1_\ell(t) \cap \DD^2_\ell(t) \cap \EE_\ell(t)
    && \forall \ell \in \Im(\tgm_{L+1}(t)),\\
    \MM^\tII_\ell(t)
    \equiv \MM^\tII_{\ell}(\tgm_{L+1}(t))&\equiv
    \A_\ell(t) \cap \widetilde{\BB}_\ell(t) \cap \CC_\ell(t)
    \cap \widetilde{\DD}_\ell(t)
    && \forall \ell \in I_2(\tgm_{L+1}(t)).
\end{align*}
The argument $t$, whose role is to indicate that $(v_{L+1},a_{L+1},t_{L+1}) = (v,a,t)$, is included as we shall soon vary $t$. However, we henceforth omit  $t$ from the notation for all events with $\ell\le L$, as they do not depend on the value of $t$.
We further write
\[
    \NN_L
    \equiv \set{\tgm_{L} \in \tGm}
    =  \Big[
        \bigcap_{\ell \in \Im} \MM^\tm_{\ell}
    \Big] \cap \Big[
    \bigcap_{\ell \in I_2} \MM^\tII_{\ell}
    \Big].
\]
By Campbell's theorem,
\begin{align*}
    \E[N^\tm_{L}\mid \tgm_{L}\in\tGm_{\min}]
    &=
    \E\Big[
        \sum_{t:(v,t,1)\in \xi}
        \Ind{\MM^\tm_{L+1}(t)}
    \ \Big|\
        \NN_{L}
    \Big]
    = \frac{1}{2}
    \int_{t_L}^\infty
    \P(\MM^\tm_{L+1}(t)\mid \NN_{L}, \A_{L+1}(t)) dt,
\end{align*}
where $N^\tm_{L}=N^\tm(\tilde \gamma_{L};v,a) $ is defined in \eqref{e:NtmL}. This motivates the following lemma.
\begin{lem}\label{l:pathrecur1}
Under the above notation,
\begin{equation*}
    \P(\MM^\tm_{L+1}(t)\mid \NN_L, \A_{L+1}(t))
    \le
    2^{-|\pd a \cap \pd a_{L}^c
    \setminus\set{v}|}
        %| + \Ind{v\in\pd a_{L}^c}}
    \cdot
    \P(\BB_{L+1}(t), \DD^2_{L+1}(t), \EE_{L+1}(t)
    \mid
        Y_{t_{L}}(\pd a_{L}) = \vec{1}
    ).
\end{equation*}
\end{lem}
Roughly speaking, the event $\MM^\tm_{L+1}(t)$ is contained in the intersection of two groups of events that are roughly independent: (1).\
 the events $\BB_{L+1}(t)$, $\DD^2_{L+1}(t)$ and $\EE_{L+1}(t)$ depending on $\pd a_L$.
 (2).\
the events $\DD^1_{L+1}(t)$ and
\begin{equation}\label{e:CCprime}
    \CC'_{L+1}(t)
    \equiv
    \bigcap_{u\in \pd a\setminus \pd a_L}
    \Big[
    \set{T_+(u;0)\le t}
    \cap
    \big(\cap_{\ell=1}^{L-1}
    \big\{T_+(u; a_\ell,t_\ell )\le t \big\}
    \big)
    \Big]
    \supseteq \CC_{L+1}(t)
    .
\end{equation}
depending on $\pd a\setminus \pd a_L$.
The first group depends on $\NN_L$ only through $ Y_{t_{L-1}}(\pd a_{L}) = \vec{1}$, whereas for the second group,   conditioning on $\CC'_{L+1}(t)$, namely that every $u\in \pd a\setminus\pd a_L$ has been updated at least once since time $0$ or its last apperance in $\tgm_L$, we intuitively expect that $\DD^1_{L+1}(t)=\{Y_t(\pd a \cap \pd a_{L}^c \setminus\set{v})= \vec{1} \} $ is roughly independent of everything else and happens with probability at most $2^{-|\pd a \cap \pd a_{L}^c \setminus\set{v}|}$.

In light of the above discussion, we  expect that (omitting $t$'s from the notation)
\begin{align*}
    \P(\MM^\tm_{L+1}\mid \NN_L, \A_{L+1})
    &\lesssim
    \P(\DD^1_{L+1}\mid \CC'_{L+1})
    \cdot
    \P(\BB_{L+1}, \DD^2_{L+1}, \EE_{L+1}
    \mid
    Y_{t_{L}}(\pd a_{L}) = \vec{1}
    )
    \\&\approx
    2^{-|\pd a \cap \pd a_{L}^c
    \setminus\set{v}|}
    \cdot
    \P(\BB_{L+1}, \DD^2_{L+1}, \EE_{L+1}
    \mid
    Y_{t_{L}}(\pd a_{L}) = \vec{1}
    )
    .
\end{align*}
However, the event $\NN_L = \set{\tgm_L\in\tGm_L}$ may depend, through events $\EE_\ell$ and $\tilde{\BB}_\ell$,  on updates at a vertex after it last appears in some hyperedges of $\tgm_L$.
While intuitively ``additional updates" and also conditioning that the restriction of the configuration to certain hyperedges at certain times will not be all $1$, ``can only help",  overcoming such dependencies is the main technical obstacle in the proof below. Through a subtle conditioning argument we will establish a positive correlation between the relevant events.
For the sake of continuity of the argument, we postponed the proof of Lemma~\ref{l:pathrecur1} to Section~\ref{s:genGlem}.

The last lemma we need before proving Lemma~\ref{l:Eminbranch} concerns random walks on hypercubes.
Its proof  is also postponed to Section~\ref{s:genGlem}.
Let $(Z_i)_{i \in \mathbb{Z}_+ }$ be the (discrete-time) lazy simple random walk on the $m$-dimensional hypercube $\set{0,1}^m$ where in each step, a coordinate is chosen uniformly at random and updated to $0$ or $1$ with equal probability.
Let $H_1 \equiv \inf\set{i>0: Z_i = \vec 1}$ be the hitting time of $\vec 1$ and let $T_+$ be the first time  by which each coordinate which equals 1 at time 0  was updated at least once.
\begin{lem}\label{l:LSRW}
     For every $m\ge 2$, the expected number of visits to  $\vec 1$ before $T_+$ satisfies
\begin{align}\label{e:Evisits}
    \E\bigg[\sum_{0\le i< T_+} \Ind{Z_i = \vec 1}\bigg]
    &\le
    \begin{cases}
        \frac{6}{m}
    & Z_0 \neq (1,1,1,\dots, 1)
    \\
    2 + \frac{6}{m}
    & Z_0 = (1,1,1,\dots, 1)
    \end{cases}
    .
\end{align}
\end{lem}

We now prove Lemma~\ref{l:Eminbranch}.
\begin{proof} [Proof of Lemma~\ref{l:Eminbranch}]
    Recall that $m_{L+1} \equiv |\pd a \cap \pd a_L\setminus\set{v}|$. By Lemma~\ref{l:pathrecur1} and the Campbell's theorem,
\begin{align}
    &2 \E[ N^\tm(\tgm_L;v,a) \mid \tgm_L\in\tGm_{\min}]
    \nonumber\\& \le
    \int_{t_L}^\infty
    2^{-|\pd a \cap \pd a_L^c\setminus\set{v}|}
    \P(\BB_{L+1}(t),
        \DD^2_{L+1}(t)\cap \EE_{L+1}(t)
    \mid
        Y_{t_{L}}(\pd a_{L}) = \vec{1}
    )
    dt
    \nonumber\\& =
    2^{-(k-1-m_{L+1})}
    \E\Big[
    \int_{t_L}^{T_+(a;a_{L},t_L)}
    \Ind{Y_t(\pd a_L)\in \Theta_{L+1}} dt
    \ \Big|\
        Y_{t_{L}}(\pd a_{L}) = \vec{1}
    \Big]
    ,\label{e:Eint}
\end{align}
where
\[
    \Theta_{L+1}
    \equiv \Theta_{L+1}(\tgm_L;v,a)
    \equiv
    \begin{cases}
        \set{\sigma \in \set{0,1}^{\pd a_L}:
        \sigma_{\pd a_{L} \cap \pd a
        } =\vec 1},
        &
        v \notin \pd a_{L}
        \\
        \set{\sigma \in \set{0,1}^{\pd a_L}:
            \sigma_{\pd a_{L} \cap
                \pd a \setminus\set{v}
            } = \vec 1
            ,
            \sigma_{\pd a_{L} \cap \pd a^c}
            \neq \vec 1
        },
        &
        v \in \pd a_{L}
    \end{cases}
\]
is the range of $Y_t(\pd a_L)$ restricted on $\DD^2_{L+1}(t)\cap \EE_{L+1}(t)$.
We now apply the result of Lemma~\ref{l:LSRW}, differentiating the three cases:
\begin{enumerate}[\quad 1.]
\item $v \notin \pd a_L$: Let $(\tilde{Y}_i)_{i \in \mathbb{Z}_+}$ be the skeleton chain (i.e.,~the chain that records the configuration of $Y_t$ on $\pd a_L \cap \pd a$ after every time a vertex in $\pd a_L \cap \pd a$ is updated) of the continuous time Markov chain $Y_t(\pd a_L \cap \pd a)$ starting from time $t_L$. Note that  $(\tilde{Y}_i)_{i\in \mathbb{Z}_+}$ is a lazy simple random walk on the $m_{L+1}$-dimensional hypercube and that the time between two steps of $\tilde{Y}_i$ in $Y_t(\pd a_L \cap \pd a)$ are i.i.d.\ random variables with $\textup{Exp}(m_{L+1})$ distribution.
    Therefore, we can rewrite the expectation on the RHS of \eqref{e:Eint} in terms of $\tilde{Y}_i$, namely,
\[
    \textup{RHS of }\eqref{e:Eint}
    =
    2^{-(k-1-m_{L+1})}
    \frac{1}{m_{L+1}}
    \E\Big[\sum_{0\le i\le \tilde{T}}
    \Ind{\tilde{Y}_i = \vec 1}
    \ \Big|\
    Y_0 = \vec 1
    \Big] ,
\]
where $\tilde{T}$ is the number of steps in $(\tilde{Y}_i)_{i \in \mathbb{Z}_+}$ until  every vertex of $\pd a_L \cap \pd a$ is updated at least once.
Applying Lemma~\ref{l:LSRW}, we get that
\[
    \textup{RHS of }\eqref{e:Eint}
    \le
    2^{-(k-1-m_{L+1})}
    \frac{6}{m_{L+1}} 
    .
\]
\item $v\in \pd a_L, a\neq a_L$:  Let  $(\tilde{Z}_i)_{i\in Z_+}$ be the skeleton chain of $Y_t(\pd a_L)$ starting from $t_L$ with $\tilde{Z}_0 = Y_{t_L}(\pd a_L) = \vec 1$ and $T_0 \equiv \min\set{i\ge 1: \tilde{Z}_i\neq \vec 1}$ be the time of the first $0$-update in $(\tilde{Z}_i)_{i\in Z_+}$.
By the construction of $\Theta_{L+1}$, for all $i \ge 1$,   if $\tilde{Z}_i\in \Theta_{L+1}$  then we must have that $i\ge T_0$.
For brevity of notation, let $A \equiv \pd a_{L} \cap \pd a\setminus\set{v}$ and define $\tilde{T}$ to be  the number of steps in $(\tilde Z_i)_{i\in \mathbb{Z}_+}$ until which every vertex in $\pd a_L \cap \pd a = A\cup \set{v}$ is updated at least once. Observe that $\tilde{T}$ corresponds to the deactivation time $T_+(a;a_L,t_L)$ in the original process.
By the strong Markov property and the total probability formula,
\begin{align*}
    \textup{RHS of }\eqref{e:Eint}
    &\le
    2^{-(k-1-m_{L+1})}
    \frac{1}{k}
    \E\Big[\sum_{T_0 \le i\le \tilde{T}}
        \Ind{\tilde{Z}_i(A) = \vec 1}
    \ \Big|\
    \tilde{Z}_0 = \vec 1
    \Big]
    \\ &\le
    2^{-(k-1-m_{L+1})}
    \frac{1}{k}
    \E\Big[
    \E\Big[\sum_{T_0 \le i\le \tilde{T}}
        \Ind{\tilde{Z}_i(A) = \vec 1}
    \ \Big|\
    \tilde{Z}_{T_0}(A) 
    \Big]
    \ \Big|\
    \tilde{Z}_0 = \vec 1
    \Big]
    .
\end{align*}
If $m_{L+1} = 0$, meaning that $v$ is the only vertex in the intersection of $a_L$ and $a$, then $\tilde{T}$ simply follows the exponential distribution of rate $1$ and
\[
    \textup{RHS of }\eqref{e:Eint}
    \le 1 \cdot 2^{-(k-1)}.
\]
For $m_{L+1}\ge 1$, one can  bound $\tilde{T}$ from above by $T_A + T_{v}$, where $T_A$ is  the number of steps  until every vertex in $A$ is updated at least once and
$$
    T_{v} \equiv
    \min\set{i > T_A:
        v \textup{ is updated at step }i}
    - T_A
$$
is the  number of additional steps  until $v$ is updated for the first time after time $T_A$. It follows that
\begin{equation}\label{e:ZiA}
    \sum_{T_0 \le i\le \tilde{T}}
        \Ind{\tilde{Z}_i(A) = \vec 1}
    \le
        \sum_{T_0 \le i\le T_A}
        \Ind{\tilde{Z}_i(A) = \vec 1}
    +
        \sum_{T_A < i\le T_A+T_{v}}
        \Ind{\tilde{Z}_i(A) = \vec 1}
    .
\end{equation}
For the second summation of \eqref{e:ZiA}, observe that $T_{v}$ follows the Geometric$(1/k)$ distribution and for any value of $\tilde{Z}_{T_0}(A)$ and $i\ge T_A$, we have that $\tilde{Z}_{i}(A)$ is uniformly distributed on $\set{0,1}^{A}$ with $|A| = m_{L+1}$. Therefore for all $\vec{z} \in \set{0,1}^{A}$
\[
    \E\Big[
        \sum_{T_A < i\le T_A+T_{v}}
        \Ind{\tilde{Z}_i(A) = \vec 1}
    \ \Big|\
    \tilde{Z}_{T_0}(A)=\vec{z}
    \Big]
    \le
    2^{-m_{L+1}} \E T_{v}
    = k2^{-m_{L+1}}.
\]
For the first sum on the RHS of \eqref{e:ZiA}, we split the discussion according to  whether $\tilde{Z}_{T_0} (A) = \vec 1$ or not. By the symmetry of the $k$ vertices of $\pd a_L$, we get that
$$\P(\tilde{Z}_{T_0} (A) = \vec 1) = (k-m_{L+1})/k.$$
Applying Lemma~\ref{l:LSRW} to the restriction of $\tilde{Z}_i$ to $A$ yields that
\[
\E\Big[
    \!\sum_{T_0 < i\le T_A}\!
    \Ind{\tilde{Z}_i(A) = \vec 1}
    \,\Big|\,
    \tilde{Z}_0 =\vec 1
    \Big]
    \le
    \Big[
        \frac{k-m_{L+1}}{k} \cdot
        {\Big(2+\frac{6}{m_{L+1}}\Big)}
        +
        \frac{m_{L+1}}{k} \cdot
        \frac{6}{m_{L+1}}
    \Big]
    \frac{k}{m_{L+1}}
,
\]
where the term $k/m_{L+1}$ is obtained via Wald's equation, by noting that the time between two updates in $A$ has a Geometric($m_{L+1}/k$) distribution. Combining all pieces together, we have that for all $1\le m_{L+1} \le k-1$,
\begin{align*}
    \textup{RHS of \eqref{e:Eint}}
    &\le 2^{-(k-1-m_{L+1})}
    \Big(
    \frac{1}{m_{L+1}}\frac{2(k-m_{L+1})+6}{k}
        +2^{-m_{L+1}}
    \Big)
    \\&\le 2^{-(k-1-m_{L+1})}
    \frac{k-m_{L+1}}{k}\frac{9}{m_{L+1}}
    .
\end{align*}
\item $a_L = a$: This is similar to the second case and for $k\ge 3$
\begin{align*}
    \textup{RHS of }\eqref{e:Eint}
    &\le
    \frac{1}{k}
    \E\Big[\sum_{0\le i\le \tilde{T}}
        \Ind{\tilde{Z}_i(A) = \vec 1}
    \ \Big|\
    \tilde{Z}_0 = \vec 1
    \Big]
    \le
    \frac{1}{k}
    \Big[
        2+ \frac{6}{k} + k2^{-(k-1)}
    \Big]
    \le \frac{5}{k},
\end{align*}
where the two terms in the third step are obtained from an  argument similar to \eqref{e:ZiA}.
\end{enumerate}
Those three cases conclude the proof with $C_\tm \equiv 9$.
\end{proof}

\subsection{Number of minimal type-II steps}\label{s:typeIIbranch}
In this section we prove Lemma~\ref{l:EtypeII}. Fix $\tgm_L\in \tGm_{M},a\in \bN(a_L),v\in \pd a$ and write $\tgm_{L+2}(t) \equiv \tgm^\tII_{L+2}(v,a,t)$. Recall the definition of $\MM^\tm_\ell$, $\MM^\tII_\ell$ and $\NN_L$ from Section~\ref{s:minbranch} and define $\MM^\tII_{L+2}(t)$ similarly.
By Campbell's theorem,
\begin{align*}
    \E[N^\tII_{L}\mid \tgm_{L}\in\tGm_{\min,M}]
    &=
    \E\Big[
        \sum_{t:(v,t,1)\in \xi}
        \Ind{\MM^\tII_{L+2}(t)}
    \ \Big|\
        \NN_{L}
    \Big]
    = \frac{1}{2}
    \int_{t_L}^\infty
    \P(\MM^\tII_{L+2}(t)\mid \NN_{L}, \A_{L+1}(t)) dt.
\end{align*}
The next lemma is the type-II analog of Lemma~\ref{l:pathrecur1}, the proof of which is postponed to Section~\ref{s:genGlem}, after the introduction of  relevant notations in the proof of Lemma~\ref{l:pathrecur1}.
\begin{lem} \label{l:pathrecur2}
    Under the notations above, for all $\tgm_L\in \tGm_{M},v\in V,a\in F$ and $t> t_{L}$,
\[
    \P(\MM^\tII_{L+2}(t)\mid \NN_{L}, \A_{L+1}(t))
    \le 2^{-(k-1)} \P(\tilde{\BB}_{L+2}(t)).
\]
\end{lem}

\begin{proof}[Proof of Lemma~\ref{l:EtypeII}]
By Lemma~\ref{l:EtypeII} and  Campbell's theorem,
\begin{align*}
    \E[N^\tII_{L}\mid \tgm_{L}\in\tGm_{\min}]
    &\le 2^{-k}
    \int_{t_L}^{\infty}
    \P(\tilde{\BB}_{L+2}(t)) dt
    = 2^{-k}
    \E[ T_+(a;T_+(a_{L};t_L))-t_{L}]
    \\&=
    2^{-k}
    \cdot
    2\E[T_+(a_L;t_L)-t_{L}] = 2^{-(k-1)}\sum_{1=1}^{k}\frac{1}{i}  \le 2^{-(k-1)} (1+\log k),
\end{align*}
where we have used the fact that the coupon collector time of $k$ coupons is $\sum_{1=1}^{k}\frac{1}{i} $.

\end{proof}
\subsection{Remaining Lemmas}\label{s:genGlem}
In this subsection we complete the proof of Lemmas~\ref{l:pathrecur1}, ~\ref{l:LSRW} and~\ref{l:pathrecur2}.
\begin{proof}[Proof of Lemma~\ref{l:pathrecur1}]
For each $u\in \pd a \cup \pd a_L$, let
\[
    \ell_-(u) \equiv
    0\vee\max\set{1\le \ell\le L:
        \ell\notin I_\star, u\in \pd a_\ell}
\]
be the last step in $\tilde \gamma_{{L+1}}$ such that the corresponding hyperedge contains $u$ and let $t_-(u) \equiv t_{\ell_-(u)}$ be the time of that step.
In particular, if $u$ has never appeared in the previous steps, then $t_{-}(u) =0$.
We consider the set
$
    S_{L+1} \equiv
    \cup_{u\in \pd a\cup \pd a_{L}}
    \set{u} \times {(} t_-(u), \infty)
    ,
$
and (recalling $Y_0(V)=\vec{1}$) let
\[
    \FF_{L+1} \equiv \FF(\tV \setminus S_{L+1})
    \equiv \sigma(\xi(\tV\setminus S_{L+1}))
\]
denote the sigma-field generated by each of the vertices in $\pd a_{L}\cup\pd a$  after time $0$ or its last appearance in $\tgm_{L+1}$ before $t$. Recall that $\NN_L
= \set{\tgm_{L} \in \tGm}$. Let
\[
    \NN^\circ_{L}\equiv\set{Y_{t_\ell}(\pd a_\ell) = \vec 1, \textup{ for all } \ell\in [L]\setminus I_\star  }\supseteq \NN_{L}.
\]
By the Markov property of process $Y_t$,
the event $\MM'_{L+1}$, defined by substituting the event $\CC_{L+1}$ in the definition of $\MM^\tm_{L+1}$ with $\CC'_{L+1}$  from \eqref{e:CCprime}, is independent of $\FF_{{L+1}}$ given $\NN^\circ_{L}$.

Meanwhile, for each $\ell\in \Im$, the first five events in the definition of $\MM^\tm_\ell$ are measurable w.r.t.\ $\FF_{L+1}$ while $\EE_\ell$ might depend also on the updates of $\xi$ in $\pd a_{\ell-1}\times [t_{\ell-1},t_\ell]$. In particular, $\EE_\ell$ is not $\FF_{L+1}$-measurable if and only if
\begin{equation}\label{e:D3r}
    \ell\in I_\EE
    \equiv \big\{\ell\in \Im:
    \EE_\ell \neq \varnothing^c
    \textup{ and }
    \exists u\in (\pd a\cap \pd a_{L}^c )\setminus \set{v},
    \ell-1 = \ell_-(u)
    \big\}
    .
\end{equation}
For each $\ell\in I_2$, the events $\A_{\ell}, \CC_{\ell}$ and $\tilde{\DD}_\ell$ are measurable w.r.t.\ $\FF_{L+1}$ while $\tilde{B}_\ell$ might depend on the updates of $\xi$ in $\pd a_{\ell-2} \times [t_{\ell-2},t_{\ell}]$. More specifically, $\tilde{B}_\ell$ is not $\FF_{L+1}$-measurable if and only if
\begin{equation}\label{e:tildeBB}
    \ell\in I_\BB
    \equiv \big\{\ell\in I_2:
    \exists u\in (\pd a\cap \pd a_{L}^c )\setminus \set{v},
    \ell-2 = \ell_-(u)
    \big\}
    .
\end{equation}
Let
\begin{equation*}
    \MM^\textup{m,F}_\ell
    \equiv
    \begin{cases}
    \MM^\tm_\ell
    & \ell \in \Im\setminus I_\EE,\\
    \A_\ell \cap \BB_\ell \cap \CC_\ell
    \cap \DD^1_\ell \cap \DD^2_\ell
    &\ell \in I_\EE
    \end{cases}
    ,\quad
    \MM^\textup{II,F}_\ell
    \equiv
    \begin{cases}
    \MM^\tII_\ell
    & \ell \in I_2\setminus I_\BB,\\
    \A_\ell \cap \CC_\ell \cap \tilde{\DD}_\ell
    &\ell \in I_\BB,
    \end{cases}
    .
\end{equation*}
be the $\FF_{L+1}$-measurable part of $\MM^\tm_\ell$ and $\MM^\tII_\ell$ and let
$
    \NN^\textup{F}_L \equiv
    (
        \cap_{\ell \in \Im} \MM^\textup{m,F}_{\ell}
    ) \cap (
        \cap_{\ell \in I_2} \MM^\textup{II,F}_{\ell}
    ).
$
We have
\[
    \P(\MM^\tm_{L+1}\mid \NN_L,\A_{L+1})
    \le \P(\MM'_{L+1}\mid \NN_L,\A_{L+1})
    = \P(\MM'_{L+1} \mid
        \NN^\textup{F}_{L},
        \A_{L+1},
        \cap_{\ell\in I_\EE} \EE_\ell,
        \cap_{\ell\in I_\BB} \BB_\ell
    ).
\]

To further simplify the conditioning part of the probability, we partition the events $\set{\EE_{\ell}}_{\ell \in I_\EE}$ and $\set{\tilde{\BB}_\ell}_{\ell\in I_\BB}$ into subsets such that each subset can be represented as the intersection of some $\FF_{L+1}$-measurable event and $\FF_{L+1}$-conditionally-independent event:
\begin{enumerate}[\quad 1.]
\item For each $\ell\in I_\EE$, we split $\pd a_{\ell-1}\cap \pd a_\ell^c$ into the  non-intersecting  union of
\[
    W_\ell \equiv
    \set{u\in \pd a_{\ell-1} \cap \pd a_\ell^c:
        u\in (\pd a\cap \pd a_{L}^c) \setminus \set{v},
        \ell_-(u) = \ell-1
    }
\]
and $ V_\ell \equiv (\pd a_{\ell-1} \cap \pd a_\ell^c)
\setminus W_\ell$.
It follows from the definition of $V_{\ell}$ that $Y_{t_\ell}(V_\ell)$ is $\FF_{L+1}$-measurable and $Y_{t_\ell}(W_\ell)$ is independent of $\FF_{L+1}$ conditioned on $\NN^\circ_{L}$. Let $\EE^0_\ell \equiv \set{Y_{t_\ell}(V_\ell) \neq \vec 1}$ and $\EE^1_\ell \equiv \set{Y_{t_\ell}(W_\ell) \neq \vec 1}$. Then  $\EE_\ell$ can be partitioned into the events $\EE^0_\ell$ and $\EE^1_\ell\setminus \EE^0_\ell = \EE^1 \cap (\EE^0)^c$.
\item For each $\ell\in I_\BB$, we similarly split $\pd a_{\ell-2} \cup \pd a_\ell$ into the  non-intersecting  union of
\[
    W_\ell \equiv
    \set{u\in \pd a_{\ell-2} \cap \pd a_\ell^c:
        u\in (\pd a\cap \pd a_{L}^c) \setminus \set{v},
        \ell_-(u) = \ell-2
    }
\]
and $ V_\ell \equiv (\pd a_{\ell-2} \cap \pd a_\ell^c) \setminus W_\ell$, and define $\tilde{W}_\ell \equiv W_\ell \times (t_{\ell-2},t_{\ell})$, $\tilde{V}_\ell \equiv V_\ell \times (t_{\ell-2},t_{\ell})$.  Recall that $\xi^\circ$ is the unmarked update sequence, i.e.,\ $(v,t)\in \xi^\circ$ if and only if $(v,t,1)\in \xi$ or $(v,t,0)\in \xi$.
The event $\tilde{\BB}_\ell$ is measurable w.r.t.\ the sigma field generated by $ \xi^\circ(\tilde{W}_\ell \cup\tilde{V}_\ell) = \xi^\circ(\tilde{W}_\ell)\times\xi^\circ(\tilde{V}_\ell)$.
More specifically, let $\Xi_\ell(\tilde{W}_\ell)$ be the  set of possible configurations of $\xi^\circ(\tilde{W}_\ell)$ in event $\tilde{\BB}_\ell$. Then
$\tilde{\BB}_\ell$ can be written as
\[
    \tilde{\BB}_\ell
    = \cup_{\xi^\circ(\tilde{W}_\ell)\in\Xi_\ell(\tilde{W}_\ell) }
    \set{\xi^\circ(\tilde{W}_\ell)}
    \times
    \set{\xi^\circ(\tilde{V}_\ell):
        \xi^\circ(\tilde{W}_\ell \cup\tilde{V}_\ell)
        \in \tilde{\BB}_\ell}
    ,
\]
where
$\xi^\circ(\tilde{W}_\ell)$ is independent of $\FF_{L+1}$ and
$\xi^\circ(\tilde{V}_\ell)$ is $\FF_{L+1}$-measurable.
\end{enumerate}
By the fundamental formula of total probability, for any two events $A,B$ and partition of $A$ into disjoint sets $A = \cup_{i=1}^n A_i$, we have
\[
    \P(B\mid A) = \sum_{i=1}^n \P(B\mid A_i) \P(A_i\mid A)
    \le \sup_{1\le i\le n} \P(B\mid A_i).
\]
Applying the same argument to the aforementioned partitions of $\EE_\ell$ and $\tilde{\BB}_\ell$, we have
\begin{align}
    &\P(\MM'_{L+1} \mid
        \NN^\textup{F}_{L},
        \A_{L+1},
        \cap_{\ell\in I_\EE} \EE_\ell,
        \cap_{\ell\in I_\BB} \BB_\ell
    )
    \nonumber
    \\& \le
    \sup_{\substack{
    \xi^\circ(\tilde{W}_\ell\cup\tilde{V}_\ell)
    \in \tilde{\BB}_\ell,
    \\
    \ell\in I_\BB
    ;
    I'_\EE\subseteq I_\EE
    }}
    \P\bigg(\MM'_{L+1} \,\bigg\vert\,
        \NN^\textup{F}_{L},
        \A_{L+1},
        \bigcap_{\ell\in I'_\EE}
        [\EE^1_\ell\cap (\EE^0_\ell)^c],
        \bigcap_{\ell\in I_\EE\setminus I'_\EE}
        \EE^0_\ell,
        \big\{\xi^\circ(\tilde{W}_\ell),
            \xi^\circ(\tilde{V}_\ell)
        \big\}_{\ell\in I_\BB}
    \bigg)
    \nonumber\\&=
    \sup_{\substack{
    \xi^\circ(\tilde{W}_\ell)
    \in \Xi_\ell(\tilde{W}_\ell)
    ,\\
    \ell\in I_\BB
    ;
    I'_\EE\subseteq I_\EE
    }}
    \P(\MM'_{L+1} \mid
        \NN^\circ_{L},
        \A_{L+1},
        \cap_{\ell\in I'_\EE}
        \EE^1_\ell,
        \set{\xi^\circ(\tilde{W}_\ell)
        }_{\ell\in I_\BB}
    )
    \nonumber\\&\le
    \sup_{\substack{
    \xi^\circ(\tilde{W}_\ell)
    \in \Xi_\ell(\tilde{W}_\ell)
    ,\\
    \ell\in I_\BB
    ;
    I'_\EE\subseteq I_\EE
    }}
    \P(\DD^1_{L+1} \mid
        \NN^\circ_{L},
        \CC'_{L+1},
        \cap_{\ell\in I'_\EE} \EE^1_\ell,
        \set{\xi^\circ(\tilde{W}_\ell)
        }_{\ell\in I_\BB}
    )
    \cdot
    \P(\BB_{L+1},\DD^2_{L+1},\EE_{L+1}
        \mid
        \NN^\circ_{L}
    )
    ,\label{e:condindep}
\end{align}
where the penultimate step uses
the conditional independency of $\MM'_{L+1}$ and $\FF_{L+1}$ given $\NN^\circ_L$,
and the last step uses the  independence of updates on $\pd a_L \times (t_L,\infty)$ and $\cup_{u\in\pd a\setminus \pd a_L} \set{u}\times (t_-(u),\infty)$.

To conclude the proof, we show that the first probability in \eqref{e:condindep} is uniformly bounded by $2^{-|\pd a \cap \pd a_L^c\setminus \set{v}|}$. Recall the definition of $W_\ell$ for each $\ell\in I_\EE \cup I_\BB$. For any $I'_\EE\subseteq I_\EE$ and $\xi^\circ(\tilde{W}_\ell) \in \Xi_\ell(\tilde{W}_\ell), \ell\in I_\BB$,  we define
\[
    W_0 \equiv W_0(I'_\EE)
    \equiv
        (\pd a\cap \pd a_{L}^c \setminus \set{v})
        \setminus (\cup_{\ell\in I'_\EE \cup I_\BB} W_\ell)
    .
\]
The events $\set{W_\ell}_{\ell\in I'_\EE \cup I_\BB \cup \set{0}}$ form a partition of set $\pd a\cap \pd a_{L}^c \setminus \set{v}$. It follows that
\begin{align*}
    \hspace{1cm}&\hspace{-1cm}
    \P(\DD^1_{L+1} \mid
        \NN^\circ_{L},
        \CC'_{L+1},
        \cap_{\ell\in I'_\EE} \EE^1_\ell,
        \set{\xi^\circ(\tilde{W}_\ell)
        }_{\ell\in I_\BB}
    )
    \\& =
    \prod_{\ell \in I'_\EE}
    \P(
        Y_{t}(W_\ell) = \vec 1
    \mid
        Y_{t_{\ell-1}}(W_\ell) = \vec 1,
        Y_{t_\ell}(W_\ell) \neq \vec 1
        ,
        T_+(a;a_{\ell-1}, t_{\ell-1}) < t
    )
    \\&
    \cdot\ \ \prod_{\ell \in I_\BB}
    \P(
        Y_{t}(W_\ell) = \vec 1
    \mid
        Y_{t_{\ell-2}}(W_\ell) = \vec 1,
        \xi^\circ(\tilde{W}_\ell)
        ,
        T_+(a; a_{\ell-2},t_{\ell-2}) < t
    )
    \\&
    \cdot \prod_{u\in W_0(I'_\EE)}
    \P(
        Y_{t}(u) = \vec 1
    \mid
        Y_{t_-(u)}(u) = \vec 1,
        T_+(u;t_-(u)) < t
    )
    .
\end{align*}
For each probability in the first product,
the monotonicity of the process $Y_t$ implies that
removing the condition of $Y_{t}(W_\ell)\neq \vec{1}$ will only increase its value. Thus
\[
    \P(\DD^1_{L+1} \mid
        \NN^\circ_{L},
        \CC'_{L+1},
        \cap_{\ell\in I'_\EE} \EE^1_\ell,
        \set{\xi^\circ(\tilde{U}_\ell)
        }_{\ell\in I_\BB}
    )
    \le \prod_{\ell\in I'_\EE \cup I_\BB \cup \set{0}}
   2^{-|U_\ell|}
    = 2^{-|\pd a \cap \pd a_L^c\setminus \set{v}|}
    .
\]
Plugging the last equation back into \eqref{e:condindep} concludes the proof.
\end{proof}

\begin{proof} [Proof of Lemma~\ref{l:LSRW}]
Denote $\vec{1} \equiv (1,1,\ldots,1)$ and  $z=(0,1,1,\ldots,1)$. We first explain how the case $Z_0 =\vec{1}  $ implies all other cases:

\bpp
For each $v \in \{0,1\}^m $, let $\mathcal{E}_v\equiv \E_{v}[|\{0 \le t \le T_+ :Z_t=\vec{1} \} |]$ be the duration of time that the process $Z_t$, starting from $v$, stays at state $\vec{1}$ before $T_+$.
By symmetry and monotonicity, $z$ achieves the maximum of $\mathcal{E}_v$ over  $v\in\set{0,1}^m\setminus\set{\vec{1}}$ (along with other maximizers). \ehl 
Let $T_+'$ be the first time by which every coordinate, apart perhaps from the first one, is updated. Denote $\mathcal{E} \equiv \mathcal{E}_{\vec{1}}$.   Then by first step analysis and symmetry
\[ \mathcal{E}=1+\frac{1}{2} \mathcal{E}_{z}+\frac{1}{2} \mathcal{E}', \quad \text{where} \quad \mathcal{E}' \equiv \E_{\vec{1}}[|\{0 \le t \le T_+' :Z_t=\vec{1} \} |].   \]
Note that $\{T_+' \neq T_+\}$ is precisely the event that the first coordinate is the last one to be updated. Given  $T_+' \neq T_+$ we have that  $T_+' - T_+$ has a Geometric($1/m$) distribution and that at each step $t$ between $T_+ $ and $T_+'$ the probability that $Z_t= \vec{1} $ is $2^{-(m-1)}$. Thus \[\mathcal{E}=\mathcal{E}'+\P_{\vec{1}}[\{T_+' \neq T_+\}]m 2^{-(m-1)}=\mathcal{E}'+2^{-(m-1)}. \]
Plugging this identity above yields that $\mathcal{E}_{z}=\mathcal{E}-2+2^{-(m-1)} $.

We now treat the case $Z_0=(1,1,\ldots,1)$. Note that after precisely $i$ coordinates have already been updated, the probability that the chain is at $\vec 1 $  is $2^{-i}$. The number of such steps follows a Geometric distribution with parameter $(m-i)/m$. Thus the desired expectation is
\[\mathcal{E}   = \sum_{i=0}^{m-1} \frac{2^{-i}m}{m-i}=m2^{-m}\sum_{i=1}^{m}\frac{2^{i}}{i}. \]
We now proceed to give an upper-bound on $\mathcal{E}$.  Let $(c_i)_{i\in\mathbb{Z_+}}$ be a sequence of real numbers and denote $d_i\equiv 2^{i}c_i $. Recall that by Abel's summation by parts formula, using the fact that $2^{i+1}-2^{i}=2^i $, we get that  for any integers $n_2\ge n_1\ge 0$,
$$\sum_{i=n_1}^{n_2}d_{i}=(d_{n_1}-d_{n_1+1}) +2d_{n_2}+\sum_{i=n_1+1}^{n_2-1}2^{i+1}(c_i-c_{i+1}) .$$
Applying the Abel's summation formula repeatedly (and noting that at each iteration the first and second term cancel out) yields that
\begin{align*}
    \bpp\sum_{i=1}^{m}\ehl\frac{2^{i}}{i}
&=\frac{2^{m+1}}{m}+\sum_{i=2}^{m-1} \frac{2^{i+1}}{i(i+1)}=\frac{2^{m+1}}{m}+\frac{2^{m+1}}{m(m-1)}+\sum_{i=3}^{m-2} \frac{2^{i+2}(2!)}{i(i+1)(i+2)}
\\&= \cdots =
\frac{2^{m}}{m}\sum_{j=0}^{\lceil m/2 \rceil -1}\frac{2}{\binom{m-1}{j}}  +\sum_{i=\lceil m/2 \rceil}^{\lfloor m/2 \rfloor +1}\frac{2^{i+\lceil m/2 \rceil -1}[(\lceil m/2 \rceil-1)!]}{i(i+1)\ldots (i+\lceil m/2 \rceil -1)}
\\& \le \frac{2^{m}}{m}
\bigg(2^{-\lceil m/2 \rceil+1_{m \in 2 \N }} + \sum_{j=0}^{\lceil m/2 \rceil -1}2\binom{m-1}{j}^{-1} \bigg)
.
\end{align*}
This yields that $\mathcal{E} \le 2 \big(\sum_{j=0}^{\lceil m/2 \rceil -1}\binom{m-1}{j}^{-1} \big)  +2^{-\lceil m/2 \rceil+1_{m \in 2 \N }}   $. Checking each case separately, it is not hard to  to verify that  for $m <7$ we have that $\mathcal{E}   \le 2+ \frac{6}{m} -2^{-(m-1)} $, whereas if $m \ge 7$ we have that $\mathcal{E}  \le 2+ \frac{2}{m-1}+\frac{10}{(m-1)(m-2)}+2^{-\lceil m/2 \rceil+1_{m \in 2 \N }} \le 2+ \frac{6}{m}-2^{-(m-1)}    $.
\end{proof}

\begin{proof}[Proof of Lemma~\ref{l:pathrecur2}]
Recall that for $v \in V$, $b \in F$ and $s \ge 0$, we have that $T_+(v;b,s) = s+ [T_+(v;s)-s]\cdot \Ind{v \in \pd b}$.
Similarly to the construction of \eqref{e:CCprime}, we define
\begin{align*}
    \CC'_{L+2}(t)
    &\equiv
    \bigcap_{u\in \pd a\setminus \pd a_L}
    \Big[
    \set{T_+(u;0)\le t}
    \cap
    \big(\cap_{\ell=1}^{L}
    \big\{T_+(u; a_\ell,t_\ell )\le t \big\}
    \big)
    \Big],
    \\
    \CC''_{L+2}(t)
    &\equiv
    \bigcap_{u\in \pd a \cap \pd a_L\setminus \set{v}}
    \big(\cap_{\ell=1}^L
    \set{T_+(u;a_\ell,t_\ell)\le t} \big)
\end{align*}
and
\[
    \tilde{\DD}^1_{L+2}(t)
    \equiv \set{Y_t(\pd a \cap \pd a_L^c\setminus \set{v})=\vec 1}
    ,\quad
    \tilde{\DD}^2_{L+2}(t)
    \equiv \set{Y_t(\pd a \cap \pd a_L\setminus \set{v})=\vec 1}.
\]
For every $u\in \pd a\cap\pd a_L$, the event $\set{T_+(u;0)\le t}$ is implied by $\set{T_+(u;a_L,t_L)\le t}$.
In particular, we have that $\CC_{L+2}(t) \subseteq \CC'_{L+2}(t) \cap \CC''_{L+2}(t)$.
Fix the choice of $t$ and suppress it from the notation.
Applying similar reasoning as in the proof of Lemma~\ref{l:pathrecur1} (and using the notation from that proof), we get that
\begin{align*}
    &\P(\MM^\tII_{L+2}\mid \NN_{L}, \A_{L+1})
    \\&\le \sup_{\substack{
    \xi^\circ(\tilde{W}_\ell)
    \in \Xi_\ell(\tilde{W}_\ell)
    ,\\
    \ell\in I_\BB
    ;
    I'_\EE\subseteq I_\EE
    }}
    \P(
        \tilde{\BB}_{L+2},
        \CC_{L+2},
        \tilde{\DD}^1_{L+2},
        \tilde{\DD}^2_{L+2}
    \mid
        \NN^\circ_{L},
        \cap_{\ell\in I'_\EE} \EE^1_\ell,
        \set{\xi^\circ(\tilde{W}_\ell)
        }_{\ell\in I_\BB}
    )
    \\&\le
    \P(\tilde{\DD}^1_{L+2}\mid
        \tilde{\BB}_{L+2},
        \CC'_{L+2}
    )
    \cdot
    \P( \tilde{\DD}^2_{L+2}
        \mid
        \tilde{\BB}_{L+2},
        \CC''_{L+2})
    \cdot
    \P(\tilde{\BB}_{L+2})
    = 2^{-(k-1)}
    \P(\tilde{\BB}_{L+2}).
\end{align*}
\end{proof}
\section{Random regular hypergraph}\label{s:regularG}
In this section we exploit the locally tree-like geometry of random regular hypergraphs and prove Theorem~\ref{t:mixing_regular}. For two vertices $v,v'\in V$, we define the distance $d(v,v')$ as the number of hyperedges on the shortest path from $v$ to $v'$. The  property we will need is the following.
\begin{dfn}
    For each $R\ge 1$, we say that $G$ is $R$-good if for every $v\in V$, the $R$-neighbourhood of $v$ (as a subgraph of the bipartite graph representation of $G$) contains at most one cycle.
\end{dfn}
A similar argument as the proof for random regular graph (i.e.,\ $k=2$) in \cite{lubetzky2010cutoff} yields the following.
\begin{ppn}[cf.\ \cite{lubetzky2010cutoff} Lemma 2.1]\label{p:Rgood}
    For any constant $R\ge 0$ and $G\sim \mathcal{H}(n,d,k)$.
\[
    \lim_{n\to\infty} \P(\textup{$G$ is $R$-good}) = 1.
\]
\end{ppn}
In the remainder of the section we are going to fix $R_\star = R_\star(\Delta,k)\ge 2$ to be a constant to be determined later and restrict our attention to the following subset of $n$-vertex hypergraphs:
\[
    \mathcal{G}\equiv \mathcal{G}_n(R_\star)
    \equiv
    \set{G: G\textup{ is $\Delta$-regular, $k$-uniform and $R_\star$-good}}.
\]
\subsection{Projected path}
We define for every $a\in F$ the subset
\[
    \cyc(a)\equiv\set{v\in\pd a:
    \textup{$v$ is contained in a cycle
        shorter than }
    2R_\star
    }
    \subseteq \pd a,
\]
where the length of cycle is the number of vertices along the cycle.  For each $G\in\mathcal{G}$,  the definition of $R_\star$-good implies that $|\cyc(a)|\le 2$ for all $a\in F$. In particular, for every $a\in F$ there is at most one $a'\in F$ such that $|\pd a \cap \pd a'|\ge 2$, in which case $\cyc(a) = \cyc(a') = \pd a\cap \pd a'$.

For each discrepancy sequence $\zeta = ((u_i,b_i,s_i))_{0\le i\le M}$, let $\gamma_L\equiv \gamma_{L(\zeta)}(\zeta) \equiv ((v_\ell,a_\ell,t_\ell))_{0\le \ell\le L} $ be the minimal path constructed according to the proof of Lemma~\ref{l:reduction} and for each $1\le \ell \le L$, let $i(\ell)$ be the step it corresponds to (via \eqref{e:EEreason}). From the construction of Lemma~\ref{l:reduction}, we can observe that for each $(v_\ell,a_\ell,t_\ell)$, $1\le \ell\le L$,
there exists an alternating sequence of vertices and hyperedges  and an increasing subsequence of indices
\[
   (\ta_0\tv_0\ta_1\tv_1\cdots\ta_{m}\tv_m),\quad
   i(\ell-1) = j_0 < j_1<\cdots<j_m = i(\ell)
\]
that ``represents'' a subsequence in $\zeta$, i.e.,\ it satisfies
\begin{equation}\label{e:updatecycle}
    \tv_r \in \pd \ta_{r} \cap \pd \ta_{r+1},
    \ 0\le r \le m-1
    ,\quad
    (\tv_{r},\ta_{r}) = (u_{j_r},b_{j_r}),
    \  0\le r\le m.
\end{equation}
For each $\ell$ such that $ a_\ell\neq a_{\ell-1}$  and $\pd a_\ell \cap \cyc(a_{\ell-1})  =\varnothing$, there must exist $v'$ such that $\pd a_\ell \cap \pd a_{\ell-1} = \set{v'}$. In order for $a_\ell$ to remain ``active'' with respect to $a_{\ell-1}$ by time $t_\ell$, namely $t_\ell \le T_+(a_\ell;a_{\ell-1},t_{\ell-1})$, there must not be any  $1\le i\le m$ such that $\tv_i = v'$. In particular, it implies that either $v_{\ell-1} = v'$ or $(\ta_0\tv_0\ta_1\tv_1\cdots\ta_{m}v')$ completes a cycle in $G$.
For $v \in \pd a \in F$ we write  $\cyc^+(a,v) \equiv \cyc(a)\cup\set{v}$. For each path $\gamma_L=((v_{\ell},a_{\ell},t_{\ell}))_{0 \le \ell \le L } \in \Gamma_L$ we define
(See Figure~\ref{fig:twobranch})
\[
\Ic \equiv
\set{1\le \ell \le L:
    \pd a_{\ell} \cap \cyc^+(a_{\ell-1},v_{\ell-1}) = \varnothing
}
\]
be the set of \emph{cycle steps} and define
$
    I_\tpd \equiv \set{1,\dots,L}\setminus\Ic
$
to be the set of \emph{direct steps}. It follows that for given $(v_\ell,a_\ell,t_\ell)$, there is at most $\Delta\cdot|\set{v_\ell}\cup\cyc(a_\ell)|\le 3\Delta$ ways to select $a_{\ell+1}$ such that the next step is a direct step.
\begin{figure}[t]
\begin{center}
\includegraphics[page = 1 ,width = 0.35\textwidth]{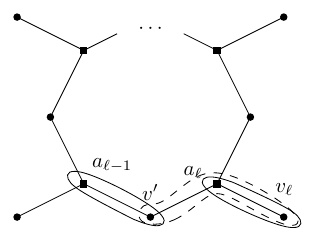}
\hspace{0.1\textwidth}
\includegraphics[page = 2 ,width = 0.35\textwidth]{twobranching}
\caption{Direct step and cycle step}\label{fig:twobranch}
\end{center}
\end{figure}

\begin{dfn}
    For each path $\gamma_L\in \Gamma_L$, we say that $\gamma_L$ is a \emph{(relaxed) projected path} if for each $1\le \ell\in L$ it satisfies the five events $\A_\ell,\BB_\ell,\CC_\ell,\DD_\ell^1,\DD_\ell^2$ defined in \eqref{e:minevents} and for each $\ell\in \Ic$, it further satisfies
\[
\GG_\ell \equiv \left\{
\begin{minipage}{0.6\linewidth}
    $\exists$ an alternating sequence
    $(\ta_0\tv_0\ta_1\tv_1\cdots\ta_{m}\tv_{m})$ such that:
    \begin{enumerate}[1.]
    \item $(\tv_0,\ta_0) = (v_{\ell-1},a_{\ell-1})$,
    $(\tv_m,\ta_m) = (v_{\ell},a_{\ell})$
    \item $(\ta_0\tv_0\ta_1\tv_1\cdots\ta_{m}v'_\ell)$ completes a cycle in $G$.
    \item $\exists t_{\ell-1}=\tlt_0\le \tlt_1<\cdots<\tlt_{m-1}\le \tlt_m = t_\ell$, such that $(\tv_r,\tlt_r)\in \xi^\circ$ for $1\le r\le m-2$.
    \end{enumerate}
\end{minipage}
\right\}
,
\]
where $v'_\ell$ is the vertex in  $\pd a_\ell \cap \pd a_{\ell-1}$.
Let $\tGm_{\proj,L}$ denotes the set of (relaxed) projected path.
\end{dfn}
\begin{rmk}\label{r:relaxedproj}
    Recall the definition in Remark~\ref{r:proj}. The discussion preceding the definition implies that $\Gamma_{\proj,L}\subseteq\tGm_{\proj,L}$. Meanwhile a path in $\tGm_{\proj,L}$ does not necessarily satisfy events~$\set{\DD^3_\ell}$. We also do not assume in the definition of $\GG_\ell$ that the values of updates at $\set{(\tv_r,\tlt_r)}$ are all ones (which turns out to be crucial in the proof). Hence the name relaxed.
\end{rmk}
\begin{rmk}
In the third requirement of the definition of $\GG_\ell$, we require $(\tv_r,\tlt_r)\in \xi^\circ$ for  $1\le r\le m-2$ so that none of the $\tv_r$'s belong to $\pd a_{\ell-1}\cup\pd a_{\ell}$ and hence $\GG_\ell$ and $\BB_\ell$ depend on different vertices. In the proof, we shall condition on events of the form $\GG_{\ell}$.  While conditioning on updates with value one works against us, conditioning on having at earlier times updates with unspecified values can only work to our advantage.
\end{rmk}

We now outline the two-step recursion on $\tGm_{\proj,L}$ and the proof \bpp of \ehl Theorem~\ref{t:mixing_regular}. Most of the arguments are parallel to the corresponding parts in Section~\ref{s:generalG}. With some abuse of notation, we occasionally override the notations in Section~\ref{s:generalG} with slightly different meanings.

Fix $\gamma_L = ((v_\ell,a_\ell,t_\ell))_{0\le \ell\le L} \in \Gamma_L$ and vertex-hyperedge pair $(v,a)$ satisfying $a\in\bN(a_L),v\in \pd a$, we let $\gamma_{L+1} (t) \equiv \gamma_{L+1}(v,a,t)$ be the path extended from $\gamma_L$ with $(v_{L+1},a_{L+1},t_{L+1}) = (v,a,t)$ and define
\begin{align*}
    N^\proj(v,a)
    &\equiv N^\proj(v,a;\gamma_{L})
    \equiv |\set{t:\gamma_{L+1}(v,a,t) \in \tGm_{\proj,L+1}}|
    .
\end{align*}

\begin{lem}\label{l:Ecyclebranch}
For any two neighbouring hyperedges $a\in F, b\in \bN(a)$, let $R(a,b)$ be the length of the shortest cycle in $G$ that contains $a$ and $b$.
There exists a constant $C_1>0$ such that
for any integer $L \ge 1$, $\gamma_L\in\tGm_{\proj,L}$ and $a\in \bN(a_L)$, $v\in \pd a$,
\begin{equation}
    \E[ N^\proj(v,a) \mid \gamma_L\in\tGm_{\proj,L}]
    \le {C_1}
    \begin{cases}
        1/k & a = a_{L}
        , \\
        2^{-k}
        & a \neq a_{L},\pd a \cap \cyc^+(a_L,v_L) \neq \varnothing
        , \\
        p_\textup{c}(R(a,a_L))
        & a \neq a_{L}, \pd a \cap \cyc^+(a_L,v_L) = \varnothing
        ,
    \end{cases}
    \label{e:Eprojstep}
\end{equation}
where
\[
        p_\textup{c}(r)\equiv
        e^{-k} + k \sum_{m\ge r}
        (\Delta k)^m
        \P(\textup{Pois}(k) \ge  m-2)
        .
\]
\end{lem}
\begin{ppn}\label{p:Rstar}
    Denote $R_\star =R_\star(\Delta,k)\equiv  \lceil e \Delta k^{2} \rceil+ 1$. Then that for all $\Delta \le 2^{k}$, $\gamma_L\in \Gamma_L$, $r\ge 2R_\star$ and $1\le \ell \le L$,
\[
p_c(r) \le (1+2^{-e \Delta k^{2} })\cdot e^{-k} \le 2e^{-k}.
\]
\end{ppn}
\begin{proof}Let $R$ follows the Poisson$(k)$ distribution of parameter $k$. Using Stirling's approximation, we have that $\P(R = m) \le \frac{e^{-k}k^m}{m!} \le \frac{e^{-k}}{\sqrt{2 \pi m}}\left( \frac{ek}{ m}\right)^m $. Hence $\P(R \ge m) \le 2  \P(R = m)  $ for   all $m \ge 2 e k$. It follows that for all $r \ge \lceil 2e \Delta k^{2} \rceil+2$,
\begin{align*}
    p_c(r) - e^{-k}
    &=
    k \sum_{m\ge r} (\Delta k)^m \P(R\ge  m-2) \le 2\Delta k^{2} \sum_{m\ge r}\frac{e^{-k}}{\sqrt{2 \pi (m-2)}}\left( \frac{e\Delta k^{2}}{ m}\right)^{m-2}
    \\&\le
    \frac{1}{2}\Delta ke^{-k} \sum_{i\ge r-2}2^{-i} \le \Delta ke^{-k}2^{-(r-2)} \le e^{-k}2^{-e \Delta k^{2} }.
\end{align*}
This concludes the proof.
\end{proof}
Meanwhile, similarly to the definition of type-II branchings in redacted paths, special treatment is needed for paths staying at the same hyperedges three steps in a row. Fix $\gamma_L = ((v_\ell,a_\ell,t_\ell))_{0\le \ell\le L} \in \Gamma_L$ and let vertices $v,v'$ satisfy $v,v'\in a_L$. Let  $\gamma_{L+2}\equiv \gamma_{L+2}(v,v',t,t') =((v_\ell,a_\ell,t_\ell))_{0\le \ell\le L+2}\in \Gamma_{L+2}$ be the path extended from $\gamma_L$ with
\[
    (v_{L+2}, a_{L+2},t_{L+2})=(v,a_L,t)
    ,\quad (v_{L+1}, a_{L+1},t_{L+1})=(v',a_L,t'),
\]
Observe that by construction, in order for $\gamma_{L+2}$ (as above) to be in $\tGm_{\proj,L+2}$, it must be the case that  $a_{L}$ is deactivated w.r.t.~itself by time $t$ (i.e.,~$t>T_+(a_L;t_L)$), making the event $Y_t(\pd a_L) = \vec{1} $ unlikely (cf.~the first case of Lemma~\ref{l:Ecyclebranch}). This is quantified in the following lemma.
We define
\[
     N^\proj_=(v,v') \equiv N^\proj_=(v,v';\gamma_L)
     \equiv|\set{(t,t'):
        \gamma_{L+2} (v,v',t,t') \in\tGm_{\proj,L+2}
    }|.
\]

\begin{lem}\label{l:Elocalbranch}
Under the notations above, there exists an absolute constant $C_2>0$ such that
\begin{equation}
    \E[ N^\proj_=(v,v';\gamma_L) \mid \gamma_L\in\tGm_{\proj,L}]
    \le {C_2}k^22^{-k}.
\end{equation}
\end{lem}
\smallskip

For $r=1,2$, Let $N_{\proj, L+r}(\gamma_L)$ be the number of paths $\gamma_{L+r}\in\tGm_{\proj, L+r}$ that agree with $\gamma_L$ in the first $L$ steps and write $N^\neq_{\proj,L+1}(\gamma_L)$ (resp.~$ N^=_{\proj,L+1}(\gamma_L)$) for the number of paths $\gamma_{L+1}$'s counted in $N_{\proj,L+1}(\gamma_L)$ that further satisfies $a_{L+1} \neq a_L$ (resp.~$a_{L+1} = a_L$).
Lemma~\ref{l:Ecyclebranch}, Lemma~\ref{l:Elocalbranch} and Proposition~\ref{p:Rstar} together imply the following theorem. The proof is presented for the sake of completeness.
\begin{thm}\label{t:Etwosteps2}
Under the above notation, there exists an absolute constant $C_\proj>0$, such that for any $R_\star$-good hypergraph $G$, integer $L\ge 1$ and path $\gamma_L\in \Gamma_L$, we have that
\begin{align}
    \E[N^=_{\proj,L+1}(\gamma_L)\mid \gamma_L\in \tGm_{\proj,L}]
    &\le
    C_\proj
    ,\label{e:Nprojeq}\\
    \E[N^\neq_{\proj,L+1}(\gamma_L)\mid \gamma_L\in \tGm_{\proj,L}]
    &\le
    C_\proj (\Delta k) 2^{-k}
    ,\label{e:Nprojneq}\\
    \E[N_{\proj,L+2}(\gamma_L)\mid \gamma_L\in \tGm_{\proj,L}]
    &\le
    C_\proj^2 [(\Delta k 2^{-k}) + k^4 2^{-k}]
    .
    \label{e:Nprojtwo}
\end{align}
\end{thm}
\begin{proof}
    Fix $\gamma_L\equiv((v_\ell,a_\ell,t_\ell))_{1\le\ell\le L} \in \Gamma_L$  and for brevity define $\tE_L\equiv\E[\,\cdot\mid \gamma_L\in\tGm_{\proj,L}]$.
  We first prove \eqref{e:Nprojeq}.  By the first case of Lemma~\ref{l:Ecyclebranch},
\[
    \tE_L[N^=_{\proj,L+1}(\gamma_L) ]
    = \sum_{v\in\pd a_L}
    \tE_L[N^\proj(v,a_L;\gamma_L)]
    \le k \cdot C_1 k^{-1}
    .
\]
We now prove \eqref{e:Nprojneq},  we define
\[
    A_{L+1}\equiv\set{a\in\bN(a_L)\setminus{ \{ a_L \}}: \pd a\cap \cyc^+(a_L,v_L) = \varnothing}
\]
be the set of hyperedges that could form a direct branching from $(v_L,a_L,t_L)$.
It follows that
\begin{align*}
    N^\neq_{\proj,L+1}(\gamma_L)
    &=
    \sum_{a\in A_{L+1}} \sum_{v\in\pd a}
        N^\proj(v,a;\gamma_L)
    +
    \sum_{a\in\bN(a_L)\setminus(A_{L+1}\cup\set{a_L})} \sum_{v\in\pd a}
        N^\proj(v,a;\gamma_L)
    .
\end{align*}
Observe that every $a\in \bN(a_L)\setminus( A_{L+1}\cup \{a_L \})$  must satisfy $R(a,a_L)\ge R_\star$. Applying the last two cases of Lemma~\ref{l:Ecyclebranch} then yields
\begin{align*}
    \tE_L[N^\neq_{\proj,L+1}(\gamma_L)]
    &\le k|A_{L+1}| \cdot C_1 2^{-k} + k|\bN(a_L)|  \cdot C_1 p_c(R_\star)
    \\&\le C_1[3k\Delta  2^{-k} + k\Delta\cdot k p_c(R^\star)]
    \le C_\proj\, \Delta k 2^{-k}
    .
\end{align*}

Finally, to prove \eqref{e:Nprojtwo}, we note that the bound of Lemma~\ref{l:Ecyclebranch} does not depend on $t_L$. Again we abbreviate $\tE_{L+1}(\cdot) \equiv \E[\cdot \mid\gamma_{L+1}\in\tGm_{\proj,L+1}]$ and let $A_{L+1}$ be defined as in the previous case (with respect to the $(L+1)$'th step). There are three possible ways of extending $\gamma_L $ to $\gamma_{L+2}$:
\[
    a_L = a_{L+1} = a_{L+2}
    \textup{\quad or \quad}
    a_L = a_{L+1} \neq a_{L+2}
    \textup{\quad or \quad}
    a_L \neq a_{L+1}
    .
\]
Following a similar argument as that of Theorem~\ref{t:Etwosteps}, we have that
\begin{align*}
    \tE[N_{\proj,L+2}(\gamma_L)]
    &\le \sum_{v,v'\in a_L} \tE_L[N^\proj_=(v,v')]
    \\&+ \sum_{a \neq a_L}\sum_{v\in\pd a}
    \tE_L[N^\proj(v,a;\gamma_L)] \cdot \tE_{L+1}[N^=_{\proj,L+2}(\gamma_{L+1}(v,a))]
    \\&+ \sum_{v\in\pd a_{L}}
    \tE_L[N^\proj(v,a_{L};\gamma_L)] \cdot \tE_{L+1}[N^\neq_{\proj,L+2}(\gamma_{L+1}(v,a_{L}))].
\end{align*}
Applying \eqref{e:Nprojeq}, \eqref{e:Nprojneq} and Lemma~\ref{l:Elocalbranch} implies that
\begin{align*}
    \tE[N_{\proj,L+2}(\gamma_L)]
&\le C_2 k^42^{-k} +
    \tE_{L}[N^\neq_{\proj,L+1}(\gamma_{L})] \max_{v,a: a \neq a_L,v \in \pd a}
    \tE_{L+1}[N^=_{\proj,L+2}(\gamma_{L+1}(v,a))]
\\&
    +
    \tE_{L}[N^=_{\proj,L+1}(\gamma_{L})]
  \max_{v \in \pd a_{L}}  \tE_{L+1}[N^\neq_{\proj,L+2}(\gamma_{L+1}(v,a_{L}))]
\\&\le
\textup{ RHS of \eqref{e:Nprojtwo},}
\end{align*}
which  concludes the proof.
\end{proof}

\begin{proof}[Proof of Theorem~\ref{t:mixing_regular}]
    The assertion of Theorem~\ref{t:mixing_regular}  is obtained by combining Lemma~\ref{l:reduction}, \eqref{e:twocases_proj}, Remark~\ref{r:relaxedproj} and \eqref{e:Nprojtwo}.
\end{proof}

\subsection{Proof of Lemma~\ref{l:Ecyclebranch} and Lemma~\ref{l:Elocalbranch}}
In this subsection we present the proofs of remaining lemmas.
We begin with Lemma~\ref{l:Ecyclebranch}.
Throughout the proof, we keep $\gamma_L\in \Gamma_{L}, a\in \bN(a_L), v\in \pd a$ fixed, and for brevity of notation write $\gamma_{L+1}(t)$ for the path extended from $\gamma_L$ with $(v_{L+1},a_{L+1},t_{L+1}) = (v,a,t)$.
We further define
\begin{align*}
    \MM^\tc_\ell(t)
    \equiv \MM^\tc_{\ell}(\tgm_{L+1}(t)) &\equiv
    \A_\ell(t) \cap \BB_\ell(t) \cap \CC_\ell(t)
    \cap \DD^1_\ell(t) \cap \DD^2_\ell(t) \cap \GG_\ell(t)
    && \forall \ell \in \Ic(\gamma_{L+1}(t)),\\
    \MM^\tpd_\ell(t)
    \equiv \MM^\tpd_{\ell}(\tgm_{L+1}(t))&\equiv
    \A_\ell(t) \cap \BB_\ell(t) \cap \CC_\ell(t)
    \cap \DD^1_\ell(t) \cap \DD^2_\ell(t)
    && \forall \ell \in I_\tpd(\gamma_{L+1}(t)),
\end{align*}
where we henceforth omit  $t$ from the notation for all events with $\ell\le L$, as they do not depend on the value of $t$.
We further write (overriding any conflicting definitions from Section~\ref{s:generalG})
\[
    \NN_L
    \equiv \set{\gamma_{L} \in \tGm_{\proj,L}}
    =  \Big[
        \cap_{\ell \in \Ic} \MM^\tc_{\ell}
    \Big] \bigcap \Big[
    \cap_{\ell \in I_\tpd} \MM^\tpd_{\ell}
    \Big].
\]
By Campbell's theorem,
\begin{align*}
    \E[N^\proj(v,a) \,|\, \gamma_{L}\in\tGm_{\proj,L}]
    &=
    \E\Big[
        \!
        \sum_{t:(v,t,1)\in \xi} \!\!
        \Ind{\MM^\bullet_{L+1}(t)}
    \,\Big|\,
        \NN_{L}
    \Big]
    = \frac{1}{2}
    \int_{t_L}^\infty
    \P(\MM^\bullet_{L+1}(t) \,|\, \NN_{L}, \A_{L+1}(t)) dt.
\end{align*}
where $\bullet = \tc$ if $L+1\in \Ic$ and $\bullet = \tpd$ otherwise.
\begin{lem}\label{l:pathrecur3}
Under the notations above,
\begin{equation*}
    \P(\MM^\tpd_{L+1}(t)\mid \NN_L, \A_{L+1}(t))
    \le
    \begin{cases}
    2^{2-k} \P(\BB_{L+1}(t))
    &a\neq a_L
    \\
    \P(\BB_{L+1}(t))
    &a= a_L
    \end{cases}
    .
\end{equation*}
Moreover if $L+1\in \Ic(\gamma_{L+1}(t))$, then
\begin{equation*}
    \P(\MM^\tc_{L+1}(t)\mid \NN_L, \A_{L+1}(t))
    \le
    2^{1-k}\P(\GG_{L+1}(t))\P(\BB_{L+1}(t)).
\end{equation*}
\end{lem}
\begin{proof}
Recall the definition of $\ell_-(u),\FF_{L+1}, \NN^\circ_{L}, \CC'_{L+1}$ from the proof of Lemma~\ref{l:pathrecur1} and omit $t$ from the notation.
By the Markov property of the  process $Y_t$, the events $\A_{L+1}$, $\BB_{L+1}$, $\CC'_{L+1}$, $\DD^1_{L+1}$, $\DD^2_{L+1}$ and (in the case of $L+1\in \Ic$) $\GG_{L+1}$ are independent of $\FF_{L+1}$ given $\NN^\circ_{L}$.
Meanwhile, for each $1\le \ell\le L$, we have that $\MM^\tpd_\ell$ is measurable w.r.t.\ $\FF_{L+1}$. It is left to treat $\set{\GG_\ell}_{\ell\in\Ic}$.

Observe that for each $\ell\in\Ic$, we have that $\GG_\ell$ is a measurable function of $\xi^\circ(V\times{[}t_{\ell-1},t_\ell))$, the time and locations of all updates between $t_{\ell-1}$ and $t_\ell$ without the value of the updates. Following a similar argument of Lemma~\ref{l:pathrecur1}, we define
\[
    W_\ell \equiv
    \set{u\in \pd a_{\ell-1} \cap \pd a_\ell^c:
        u\in (\pd a\cap \pd a_{L}^c) \setminus \set{v},
        \ell_-(u) = \ell-1
    }
\]
and (overriding the definition in Lemma~\ref{l:pathrecur1})
\[
\tilde{W}_\ell \equiv W_\ell\times{[}t_{\ell-1}, t_\ell)
, \quad
\tilde{V}_\ell \equiv W_\ell^c\times{[}t_{\ell-1}, t_\ell)
.
\]
Let $\Xi_\ell(\tilde{W}_\ell)$ be the range of the unmarked update process $\xi^\circ(\tilde{W}_\ell)$ over $\GG_\ell$ (i.e.,~the collection of all possible values of $\xi^\circ(\tilde{W}_\ell)$, provided that $\GG_{\ell}$ occurs).
It follows that
\[
    \GG_\ell
    = \cup_{\xi^\circ(\tilde{W}_\ell)\in\Xi_\ell(\tilde{W}_\ell) }
    \set{\xi^\circ(\tilde{W}_\ell)}
    \times
    \set{\xi^\circ(\tilde{V}_\ell):
        \xi^\circ(\tilde{W}_\ell \cup\tilde{V}_\ell)
        \in \GG_\ell}
    ,
\]
where $\xi^\circ(\tilde{W}_\ell)$ is independent of $\FF_{L+1}$ and $\xi^\circ(\tilde{V}_\ell)$ is $\FF_{L+1}$-measurable.
Therefore following a similar calculation to \eqref{e:condindep}, we have
\begin{align}
    &\P(\MM^\tpd_{L+1}(t)\mid \NN_L,A_{L+1}(t))
\nonumber\\&\le \sup_{\xi^\circ(\tilde{W}_\ell)\in\Xi_\ell(\tilde{W}_\ell) }
    \P(\DD^1_{L+1}\mid
        \NN^\circ_L,\CC'_{L+1},
        \set{\xi^\circ(\tilde{W}_\ell)}_{\ell\in\Ic})
    \cdot \P(\BB_{L+1},\DD^2_{L+1}\mid \NN^\circ_L)
    \nonumber\\&\le
    2^{-(k-2)\Ind{a\neq a_L}} \P(\BB_{L+1})
    ,
    \label{e:MMd}
\end{align}
where for the last step, we note that for $G\in\mathcal{G}$, we have that $|\pd a \cap \pd a_L|\le 2$ if $a\neq a_L$.

Recall that the definition of $\GG_{L+1}$ does not involve the updates on $\pd a_L \cup \pd a$. The result for $\MM^\tc_{L+1}$ follows a similar argument to that of \eqref{e:MMd}.
\end{proof}

\begin{proof}[Proof of Lemma~\ref{l:Ecyclebranch}]
    The first two cases follow from a similar argument to that of Lemma~\ref{l:Eminbranch} with overlap $m_{L+1}=k$ and $m_{L+1}\le 2$, respectively. Here we only present the proof of the third case, leaving the first two as an exercise. Fix $(v,a)$ such that $a\neq a_L$ and $\pd a\cap \cyc^+(a_L,v_L) = \varnothing$. Let $\gamma_{L+1} = \gamma_{L+1}(v,a;\gamma_L)$. We can write
\begin{align}
    \int_{t_L}^\infty \P(\GG_{L+1}(t))\P(\BB_{L+1}(t))  dt
    &\le \int_{t_L}^{t_L+ k} \P(\GG_{L+1}(t))dt
    +\int_{t_L+k}^{\infty} \P(\BB_{L+1}(t)) dt.
    \nonumber\\&\le
    k \P(\GG_{L+1}(t_L+ k))
    + \P(T_{+}(a_{L+1};a_L,t_L)>t_L+k)
    \label{e:GGBB}
\end{align}
Since in a cycle step $|\pd a_L\cap \pd a_{L+1}| = 1$, the second term on the RHS of \eqref{e:GGBB} can be bounded by $e^{-k}$.
For the first term, we enumerate over all possible cycles containing $a_L,a_{L+1}$. Fix some cycle $(\ta_0\tv_0\ta_1\dots\ta_m\tv_m)$ with $\tilde a_0=a_L$ and $\tilde a_m=a_{L+1} $. Let $s_0\equiv t_L $ and inductively define $s_{i} = T(\tilde v_{i};s_{i-1})$, for all $1 \le i \le m-2$.  Denote $\delta_i\equiv s_{i}-s_{i-1} $ and $S\equiv \sum_{i=1}^{m-2}\delta_i $. Then $\delta_1,\ldots,\delta_{m-2}$ are i.i.d.~Exp($1$) r.v.'s. In particular, interpreting the $\delta_i$'s as spacings between arrivals of a rate 1 Poisson process, we get that $\P(S \le k)= \P(N_k \ge m-2) $, where    $N_k$ has a Poisson distribution of parameter $k$. In conclusion,
\[
    \P(\tv_1,\dots,\tv_{m-2} \textup{ are sequentially updated during } (t_L,t_L+k))
=\P(S \le k )   = \P(N_k \ge m-2).
\]
Noting that there are at most $(\Delta k)^m$  cycles of length $m$ containing $a_L$ finishes the proof.
\end{proof}

We now prove Lemma~\ref{l:Elocalbranch}. Fix $\gamma_L\in \Gamma_L$ and $v,v'\in a_L$.  We define the events $\MM^\tpd_{\ell}$, $\MM^\tpd_{L+1}(t')$ and $\MM^\tpd_{L+2}(t)$  in a similar fashion as $\MM^\tpd_{\ell}$, $\MM^\tpd_{L+1}(t)$ in the proof of Lemma~\ref{l:Ecyclebranch}. Observe that  any $\gamma_{L+2} \equiv \gamma_{L+2}(v,v',t,t')$, must satisfy that $L+1,L+2\in I_\tpd$. By Campbell's theorem,
\begin{align*}
    \E[N^\proj_2\mid \gamma_{L}\in\tGm_{\proj,L}]
    &=
    \E\Big[
        \sum_{\substack{t> t':(v,t,1),(v',t',1)\in \xi}}
        \Ind{\MM^\tpd_{L+1}(t')}
        \cdot
        \Ind{\MM^\tpd_{L+2}(t)}
    \ \Big|\
        \NN_{L}
    \Big]
    \\&= \frac{1}{4}
    \int_{t_L}^\infty
    \int_{t_L}^t
    \P(\MM^\tpd_{L+2}(t), \MM^\tpd_{L+1}(t')
    \mid \NN_{L}, \A_{L+1}(t'), \A_{L+2}(t))
    dtdt'.
\end{align*}
\begin{proof}[Proof of Lemma~\ref{l:Elocalbranch}]
    Following a similar argument of Lemma~\ref{l:pathrecur3} we can show that for each $\gamma_{L+2}(v,v',t,t')\in\Gamma_{L+2}$,
\begin{align}
    &\P(\MM^\tpd_{L+2}(t), \MM^\tpd_{L+1}(t')
    \mid \NN_{L}, \A_{L+1}(t'), \A_{L+2}(t))
    \nonumber\\&\le \P(\BB_{L+1}(t'),\BB_{L+2}(t))
    \cdot \P(\DD^2_{L+2}(t)
        \mid \NN^\circ_L, \CC'_{L+2}(t), \BB_{L+1}(t'),\BB_{L+2}(t)))
    \nonumber\\&\le 2^{1-k}\P(\BB_{L+1}(t'),\BB_{L+2}(t))
    ,
    \label{e:BLBL}
\end{align}
where in the second step we ignored the event $\DD^2_{L+1}(t')$ and in the last step we used the independency between $Y_{t}(a_L)$ and $Y_{t_L}(a_{L})$ given $\CC'_{L+2}(t)$.

Now let $T_1 \equiv T_+(a_L;a_L,t_L)$ and $T_2 \equiv T_+(a_L;a_L,T_1)$. By monotonicity of deactivation time, event $\BB_{L+2}(t)\subseteq\set{T_2>t}$. Integrating the RHS of \eqref{e:BLBL} over $t,t'$, we have
\[
    \E[N^\proj_2\mid \gamma_{L}\in\tGm_{\proj,L}]
    \le 2^{-(k+1)}
    \int_{t_L}^\infty
    (t-t_L)\P(T_2>t)
    dt
    = 2^{-(k+1)} \E (T_2-t_L)^2
    .
\]
Both $(T_1-t_L)$ and $(T_2-T_1)$ are distributed as the maximum of $k$ i.i.d.\ Exp(1) random variables and they are independent with each other. Therefore a very crude bound gives
\[
    \E (T_2-t_L)^2 \le 4\E(T_1-t_L)^2
    \le 4k^2 \E_{X\sim\textup{Exp}(1)}[X^2]
    \le 8k^2
    .
\]
Plugging the last inequality into \eqref{e:BLBL} concludes the proof.
\end{proof}

\section{From sampling to counting}\label{s:FPRAS}
In this section we derive Corollary \ref{t:FPRAS} from our main result.  The corollary follows from the rapid mixing of the Markov chain and following lemma, which is an analog of \cite[Appendix A]{liu2015fptas} and \cite[Lemma~5]{Bezakova2015}. Let $\mathcal{G}(k,\Delta)$ be the set of $k$-uniform hypergraphs of maximal degree $\Delta$.

\begin{lem}\label{l:FPRAS}
    Let $k$ and $\Delta$ be positive integers and $\mathcal{G}\subseteq \mathcal{G}(k,\Delta)$ be a subset of $\mathcal{G}(k,\Delta)$ that is closed under removal of hyperedges. Suppose that for each $\Delta, k$, there is a polynomial-time algorithm (in $n$ and $1/\epsilon$) that takes a hypergraph $G=(V,F,E)\in \mathcal{G}$ with at most $n$ vertices, a vertex $v\in V$ and an $\epsilon>0$ and outputs a quantity $p(v;G)$ satisfying
\[
    \left|
        \frac{p(v;G)}
        {\P_{G}(\sigma_v = 0)}
        - 1
    \right|
    < \epsilon
    .
\]
with probability $1-\epsilon/n$, where $\sigma$ is a uniformly sampled independent set on $G$. Then there exists an FPRAS which approximates $Z(G)$ for all hypergraphs in $\mathcal{G}$.
\end{lem}

\begin{proof}
The proof is a slight modification from the argument in \cite[Lemma~5]{Bezakova2015} which we only include here for the sake of completeness.
Fix $\epsilon > 0$ and $G=(V,F,E)\in \mathcal{G}\subseteq \mathcal{G}(k,\Delta)$. Without loss of generality, we suppose $V\equiv [n] \equiv \set{1,\dots,n}$. Let $G_0 \equiv G$ and for each $1\le i\le n-1$, let $G_{i}$ be the remaining hypergraph after removing the first $i$ vertices  $ [i] \equiv \set{1,\dots,i}$ along with all hyperedges containing  at least one vertex in $[i]$.  The set of independent sets on $G_i$ can be naturally identified with the subset of independent sets on $G$ satisfying $\sigma\vert_{\set{1,\dots,i}} =  \vec 0$. We have
\begin{align*}
    \frac{1}{Z(G)}
    &= \P_G(\vec \sigma = \vec 0)
    = \P(\sigma_{1} = 0)
        \prod_{i=2}^{n}
        \P_G(
            \sigma_{i}=0
        \mid
        \sigma_{[i-1]}=0
        )
    = \prod_{i=1}^{n}
    \P_{G_{i-1}}(\sigma_{i} = 0).
\end{align*}
By assumption, the set $\mathcal{G}$ is closed under the removal of hyperedges, thus if $G\in \mathcal{G}$, then so is every $G_i$, for all $1\le i \le n-1$. Consequently, we can compute (in $\textup{poly}(n,1/\epsilon)$ time) quantities $p_i\equiv p(i;G_{i-1})$ such that
\[
    \left|
        \frac{p(i;G_{i-1})}
        {\P_{G_{i-1}}(\sigma_i = 0)}
        - 1
    \right|
    < \frac{\epsilon}{2n}
    ,
\]
with probability $1-\epsilon/n$.
Letting $\hat{Z}(G) \equiv \prod_{i=1}^n p(i;G_{i-1})$ be the output concludes the proof.
\end{proof}

\begin{proof}[Proof of Corollary~\ref{t:FPRAS}] 
    Following the statement of Lemma~\ref{l:FPRAS}, we set $\mathcal{G} = \mathcal{G}(k,\Delta)$ and describe the $p(v,G)$-outputting algorithm as follows: Given hypergraph $G$ and  $n,\epsilon>0$, let $\tmix=O(n\log n)$ be the mixing time of the Glauber dynamics of hypergraph independent set on $G$ and let $N,M$ be two large integers to be determined shortly. We run the Glauber dynamics $N$ times for $M\cdot \tmix$ steps, starting from the all zeros configuration,  and record the configuration at time $M\cdot \tmix$  of the $r$'th sample by $\sigma^{(r)}$.
We set $M \equiv1+ 2\lceil|\log\epsilon|/\log 2 \rceil$. By the submultiplicity property $\tmix(2^{-i}) \le i  \tmix $ \cite[page 55]{levin2009markov} we have that
\[
    \big|
        \P(\sigma^{(1)}_v=0) - \P_G(\sigma_v = 0)
    \big|
\le2 \|\P(\sigma^{(1)}=\cdot) - \P_G(\sigma = \cdot) \|_{\mathrm{TV}}    \le 2^{-(M-1)} \le \epsilon^2 < \epsilon/4
    ,
\]
where $\sigma$ is a uniformly chosen independent set. We set $N \equiv 32\lceil | \log\epsilon | \rceil/\epsilon^2$ and
$
    p(v;G) \equiv
    \frac{1}{N} \sum_{r=1}^N
    \Ind{\sigma^{(r)}_v = 0}
    .
$
 By Azuma-Hoeffding's inequality,
\[
    \P\big(|p(v;G)-\P(\sigma^{(1)}_v=0)| > \epsilon/4 \big)
    \le e^{-N \epsilon^2/32}
    \le \epsilon.
\]
Note that $\P_G(\sigma_v = 0)\ge 1/2$ for any hypergraph $G=(V,F)$ and all $v \in V$.  Combining the last two displays then guarantees that $|p(v;G)/\P(\sigma_v=0)-1| < \epsilon$ with probability $1-\epsilon$. The total running time of our algorithm is $N\cdot M \cdot \tmix$, which by Theorem~\ref{t:mixing} is $\textup{poly}(n,1/\epsilon)$.
\end{proof}
 
\bibliographystyle{plain}
\bibliography{refs}
\end{document}